\documentclass[11pt]{amsart}
\usepackage{amscd}
\usepackage[arrow,matrix]{xy}
\usepackage{graphicx}
\usepackage{amsmath}
\usepackage{amsmath, latexsym, amssymb}
\input xypic
\numberwithin{equation}{section}
\theoremstyle{plain}
\newtheorem{lemma}{Lemma}[section]
\newtheorem{proposition}[lemma]{Proposition}

\newtheorem{corollary}[lemma]{Corollary}

\theoremstyle{definition}
\newtheorem{definition}[lemma]{Definition}
\newtheorem{remark}[lemma]{Remark}
\newtheorem{example}[lemma]{Example}

\begin{document}
\newcommand{\R}{{\mathbb R}}
\newcommand{\C}{{\mathbb C}}
\newcommand{\F}{{\mathbb F}}
\renewcommand{\O}{{\mathbb O}}
\newcommand{\Z}{{\mathbb Z}} 
\newcommand{\N}{{\mathbb N}}
\newcommand{\Q}{{\mathbb Q}}
\renewcommand{\H}{{\mathbb H}}

\newcommand{\Aa}{{\mathcal A}}
\newcommand{\Bb}{{\mathcal B}}
\newcommand{\Cc}{{\mathcal C}}    
\newcommand{\Dd}{{\mathcal D}}
\newcommand{\Ee}{{\mathcal E}}
\newcommand{\Ff}{{\mathcal F}}
\newcommand{\Gg}{{\mathcal G}}    
\newcommand{\Hh}{{\mathcal H}}
\newcommand{\Kk}{{\mathcal K}}
\newcommand{\Ii}{{\mathcal I}}
\newcommand{\Jj}{{\mathcal J}}
\newcommand{\Ll}{{\mathcal L}}    
\newcommand{\Mm}{{\mathcal M}}    
\newcommand{\Nn}{{\mathcal N}}
\newcommand{\Oo}{{\mathcal O}}
\newcommand{\Pp}{{\mathcal P}}
\newcommand{\Qq}{{\mathcal Q}}
\newcommand{\Rr}{{\mathcal R}}
\newcommand{\Ss}{{\mathcal S}}
\newcommand{\Tt}{{\mathcal T}}
\newcommand{\Uu}{{\mathcal U}}
\newcommand{\Vv}{{\mathcal V}}
\newcommand{\Ww}{{\mathcal W}}
\newcommand{\Xx}{{\mathcal X}}
\newcommand{\Yy}{{\mathcal Y}}
\newcommand{\Zz}{{\mathcal Z}}

\newcommand{\zt}{{\tilde z}}
\newcommand{\xt}{{\tilde x}}
\newcommand{\Ht}{\widetilde{H}}
\newcommand{\ut}{{\tilde u}}
\newcommand{\Mt}{{\widetilde M}}
\newcommand{\Llt}{{\widetilde{\mathcal L}}}
\newcommand{\yt}{{\tilde y}}
\newcommand{\vt}{{\tilde v}}
\newcommand{\Ppt}{{\widetilde{\mathcal P}}}
\newcommand{\bp }{{\bar \partial}} 

\newcommand{\Remark}{{\it Remark}}
\newcommand{\Proof}{{\it Proof}}
\newcommand{\ad}{{\rm ad}}
\newcommand{\Om}{{\Omega}}
\newcommand{\om}{{\omega}}
\newcommand{\eps}{{\varepsilon}}
\newcommand{\Di}{{\rm Diff}}
\newcommand{\Pro}[1]{\noindent {\bf Proposition #1}}
\newcommand{\Thm}[1]{\noindent {\bf Theorem #1}}
\newcommand{\Lem}[1]{\noindent {\bf Lemma #1 }}
\newcommand{\An}[1]{\noindent {\bf Anmerkung #1}}
\newcommand{\Kor}[1]{\noindent {\bf Korollar #1}}
\newcommand{\Satz}[1]{\noindent {\bf Satz #1}}

\renewcommand{\a}{{\mathfrak a}}
\renewcommand{\b}{{\mathfrak b}}
\newcommand{\e}{{\mathfrak e}}
\renewcommand{\k}{{\mathfrak k}}
\newcommand{\pg}{{\mathfrak p}}
\newcommand{\g}{{\mathfrak g}}
\newcommand{\gl}{{\mathfrak gl}}
\newcommand{\h}{{\mathfrak h}}
\renewcommand{\l}{{\mathfrak l}}
\newcommand{\sm}{{\mathfrak m}}
\newcommand{\n}{{\mathfrak n}}
\newcommand{\s}{{\mathfrak s}}
\renewcommand{\o}{{\mathfrak o}}
\newcommand{\so}{{\mathfrak so}}
\renewcommand{\u}{{\mathfrak u}}
\newcommand{\su}{{\mathfrak su}}
\newcommand{\ssl}{{\mathfrak sl}}
\newcommand{\ssp}{{\mathfrak sp}}
\renewcommand{\t}{{\mathfrak t }}
\newcommand{\Cinf}{C^{\infty}}
\newcommand{\la}{\langle}
\newcommand{\ra}{\rangle}
\newcommand{\half}{\scriptstyle\frac{1}{2}}
\newcommand{\p}{{\partial}}
\newcommand{\notsub}{\not\subset}
\newcommand{\iI}{{I}}               
\newcommand{\bI}{{\partial I}}      
\newcommand{\LRA}{\Longrightarrow}
\newcommand{\LLA}{\Longleftarrow}
\newcommand{\lra}{\longrightarrow}
\newcommand{\LLR}{\Longleftrightarrow}
\newcommand{\lla}{\longleftarrow}
\newcommand{\INTO}{\hookrightarrow}
\newcommand{\supp}{\operatorname{supp}}
\newcommand{\QED}{\hfill$\Box$\medskip}
\newcommand{\UuU}{\Upsilon _{\delta}(H_0) \times \Uu _{\delta} (J_0)}
\newcommand{\bm}{\boldmath}

\title[Smooth structures on conical pseudomanifolds ]{\large Smooth structures on pseudomanifolds 
with isolated conical singularities}
\author{H\^ong V\^an L\^e, Petr Somberg and Ji\v r\'i Van\v zura} 
\thanks{H.V.L. and J.V. are supported in part by  RVO:6798540, P.S. and J.V. were supported in part by MSM 0021620839 and GACR 201/08/0397}

\medskip

\abstract  In this note we introduce the notion of a smooth structure on a conical pseudomanifold  $M$  in terms of   $C^\infty$-rings of smooth functions on $M$. For  a finitely generated smooth structure  $C^\infty  (M)$ we  introduce the notion of  the Nash tangent bundle,  the Zariski tangent bundle,  the  tangent bundle  of $M$, and  the notion
of characteristic classes of  $M$.  We   prove the vanishing of a Nash vector field  at a  singular point  for a special class of  Euclidean smooth structures on $M$.  We introduce the notion of  a conical symplectic  form on $M$ and show that
it is smooth with respect to a  Euclidean smooth structure on $M$. If  a conical  symplectic structure is also smooth with respect to  a 
compatible Poisson smooth structure $C^\infty (M)$, we show that its Brylinski-Poisson homology    groups coincide with the de Rham homology  groups of $M$.  We show nontrivial examples of  these  smooth  conical  symplectic-Poisson pseudomanifolds.
\endabstract

\maketitle
\tableofcontents

{\it  AMSC: 51H25, 53D05, 53D17}


\section{Introduction}

Since the second half of the last century the theory of smooth manifolds has been  extended from various points of view to a large class of topological spaces admitting
singularities, see e.g. \cite{Dubuc1981}, \cite{GS2003}, \cite{Joyce2009}, \cite{Kock2006}, \cite{MR1991}, \cite{Mostow1979},  \cite{SL1991}, \cite{Smith1966}. Roughly speaking,  a  $C^k$-structure, $1\le k \le \infty$, on a topological space $M$ is defined by a choice of a subalgebra  $C^k(M)$ of the $\R$-algebra  $C^0 (X)$ of all continuous $\R$-valued functions  on $M$, which satisfies certain axioms varying in different  approaches.   Most of efforts have been spent on construction of a convenient category of smooth spaces, which should satisfy good formal properties, see \cite{BH2008} for a survey. Notably, the theory of  de Rham cohomology has
been extended to a large class of singular spaces, see \cite{Mostow1979}, \cite{Smith1966}.  

In this note we develop the  theory of smooth structures on  singular spaces in a different direction.
We pick a  class of topological spaces and ask, if we can  provide   these spaces with  a family of reasonable  smooth
structures and what is the best smooth structure on    a singular space.  This question is motivated by the   question of finding  the best compactification  of an  open smooth manifold. We are looking not only for an extension of  classical  theorems on smooth manifolds, but we are also looking for new phenomena  on these manifolds, which are caused
by presence of   nontrivial singularities.

We study in this note   pseudomanifolds with  isolated conical singularities. Our choice is motivated by the following reasons.  Firstly, isolated conical singularities are geometrically the simplest possible, but they already serve to illustrate new phenomena that are typical for the more general situation.   Secondly, the theory    of smooth structures on singular spaces should include  investigations related to different geometric structures  compatible  with   these smooth structures.  A  closely  related  field of research has been developed since  Cheeger wrote  the seminal paper on    spectral geometry of   Riemannian spaces with isolated conical singularities \cite{Cheeger1979}.  We would like to  emphasize that Cheeger and   other people working on spectral geometry and index theory on singular spaces, e.g. \cite{Cheeger1983}, \cite{DLN2008}, \cite{Lesch1996},  deal with the  analysis on the open
regular strata  $M^{reg}$ of  a  compact singular space  $M$.  Although $M^{reg}$ is
open, for a large class of spaces the compactness of $M$ forces
the most fundamental features of the theory on compact
manifolds to continue to hold for $M^{reg}$.   They did not consider $M$ as a  smooth space.
 
The plan of our note is as follows.  In   section  2 we introduce the notion of a pseudomanifold  $M$  with isolated conical singularities, which we  will  abbreviate as {\it a pseudomanifold w.i.c.s.}, their cotangent bundles,  the notion of  smooth functions, smooth differential forms on these spaces  and the notion of  smooth mappings between these spaces.   Known and new examples  are given (Example \ref{res}), some important properties of these smooth structures are  proved (Lemma \ref{main1}, Proposition \ref{fin},  Corollary \ref{germ}), which are important in later sections. Our approach is  close to the  approach by
Mostow in \cite{Mostow1979}, which is  formalized in the theory of $C^\infty$-rings  as in \cite{MR1991}.  Roughly speaking,  a smooth structure on a  pseudomanifold w.i.c.s. is specified by  the canonical smooth structure on its regular stratum and a  smooth structure  around its singular points, which dictates the way to ``compactify" the smooth structure  around the singular point 
(Definition \ref{smooth}).  In  section 3 we consider  finitely generated  smooth structures. We introduce  different notions of  tangent bundles of a  smooth pseudomanifold  w.i.c.s.,  which leads to the notion of characteristic classes of a  finitely generated smooth pseudomanifold  w.i.c.s. (Remark \ref{char}). We  investigate    some properties of   related smooth vector fields (Proposition  \ref{dynamic},  Lemma \ref{zaris}.2).   
 In section 4 we introduce the notion of a conical symplectic form on a   pseudomanifold  w.i.c.s. $M$. We show that  this  symplectic form is smooth with respect to a  Euclidean smooth structure   on $M$  (Corollary \ref{smoothom}), and it possesses  a unique up to homotopy  compatible  $C^1$-smooth conical Riemmanian metric  (Lemma \ref{comp}).   We also show that if this  conical symplectic form  is compatible  with a    Poisson smooth structure on $M$,      the symplectic homology of $M$ is well-defined  (Remark \ref{SLP}, Lemma \ref{bryl}). If the symplectic form is also smooth with respect to the compatible Poisson  smooth structure, we prove that the  Brylinski-Poisson homology
coincides with    the de Rham cohomology of $M$  with reverse grading  (Corollary \ref{hodge}).  
In Remark \ref{SLP}  we show non-trivial examples  of   smooth Poisson structures  which are compatible with a conical  symplectic structure  on $M$.
In  section  5 we summarize our  main results, and  pose some  questions for further investigations.

\section{Pseudomanifolds  w.i.c.s. and their smooth structures}

In this section we introduce the notion of  a pseudomanifold w.i.c.s.  $M$, their cotangent bundles,  the notion of  a smooth structure, smooth differential forms on these spaces  and the notion of  smooth mappings between these spaces.
     We provide  known and new examples,  representing   the algebra of  smooth functions in terms  of generators and relations  (Example \ref{res}, Lemma \ref{wcone}.1,  Remark \ref{exotic}). We compare our concepts
with some existing concepts.  We prove some important properties of  these smooth structures  (Lemma \ref{main1}, Proposition \ref{fin},  Corollary \ref{germ}, Lemma \ref{wcone}.2), 
and we characterize the removability of a singular point $s\in M$ in terms of  the local  algebra of smooth functions on  a neighborhood   $\Nn_s$ of $s$  (Lemma \ref{opt}).

If $L$ is  a smooth manifold, the cone over $L$  is the topological space
$$cL : = L \times [0, \infty)\, /L\times \{0\}.$$
The image  of $L \times \{ 0\}$ is the singular point of    cone $cL$. Let $[z,t]$ denote the image of $(z,t)$ in $cL$
under the projection $\pi: L \times [0,\infty) \to cL$.  
Let $\rho_{cL} : cL \to [0,\infty)$ be defined by $ \rho_{cL}([z, t]): = t$. We call $\rho_{cL}$ the { \it defining function of the cone}.
For any $\eps  > 0$ we denote by $cL (\eps)$  the open subset $\{ [z, t] \in cL| \:  t< \eps \}$.

\begin{definition}  (cf. \cite[Definition 1.1]{DLN2008}) A   second-countable locally compact Hausdorff topological space  $M^m$  is called {\it  a    pseudomanifold  with isolated conical singularity}  of dimension $m$, if  there is a finite set $S$ (or $S_{M^m}$) of  isolated singular points  $s_i \in M^m$ such that:

1. $M^m \setminus \cup _i \{ s_i \}$ is an open smooth manifold  $M^{reg}$ of dimension $m$.

2. For each  singular point $s$  there is  an open neighborhood $\Nn_s$ of $s$  together  with a   homeomorphism
$\phi_s : \Nn _s \to  cL_s(\eps_s)$, where $L_s$ is a closed smooth manifold, $\eps_s >0$, and  the restriction of
$\phi_s$ to $\Nn_s \setminus \{ s\}$ is a smooth diffeomorphism on its image. 

3. If $s_0 , s_1 \in S$, then either  $\Nn_{s_0} \cap \Nn_{s_1} = \emptyset $, or $s_0 = s_1$.
\end{definition}

 The  smooth manifold $M^{reg} : = M^m \setminus S$  is called  the {\it regular stratum} of $M^m$, and $L_s$ (or 
 simply $L$) is called
{\it the singularity link} of a singular point $s$.
The map $\phi_s: \Nn_s\to cL(\eps_s)$ is called {\it a singular chart} (around a singular point $s$).  We also denote by $\Nn_s (\eps)$ the preimage $\phi_s ^{-1} (cL(\eps))$ for  $0<\eps\leq\eps _s$.

For the simplicity of exposition we
assume in this note that $M^{reg}$ and $L_s$ are orientable and $M^{reg}$ is connected.

\begin{example}\label{exam1} 1. Let $M$ be a  smooth manifold with  boundary $\p M = L$  which is a disjoint union of
$k$ compact connected components  $L_i$.  An easy way  to construct a   pseudomanifold w.i.c.s. is to glue to  $M$  the closed  cone 
$\bar cL: = L \times [0,1]/L \times \{0\}$, or  to glue to $M$  the union of  the closed cones
$\bar cL_i$ along the boundary  $\p M = L \times \{ 1\} = \cup_i (L_i \times \{1\})$.

2. The quadric $Q _m =\{ z\in \C^{m+1}|\, \sum _{i =1} ^{m +1}z_i ^ 2 = 0\}$ with isolated singularity at $0$ is a  pseudomanifold  w.i.c.s..    The quadric $Q_m $    is  a cone over $L = Q_m \cap S^{2m+1}(\sqrt 2)$,
where $S^{2m+1}(\sqrt 2)$ is  the sphere of radius  $\sqrt 2$ in $\C^{m+1}$.  It is easy to see that $L$ consists of  all pairs $(x, \sqrt{-1} y) \in  S^m (1)\times S^{m}(1) \subset\R^{m +1} \oplus \sqrt{-1}\R^{m +1} = \C^{m +1}$ such that $ \la x, y \ra  = 0$. Hence
$L$ is diffeomorphic  to the  real Stiefel manifold $V_{2,m +1} = SO (m+1)/SO (m -1)$.

3.  Any smooth manifold with $k$   marked points   is a  pseudomanifold  w.i.c.s. with singular points  being the marked points.
\end{example}

Now let us  introduce the notion of a smooth structure on a  pseudomanifold w.i.c.s. by refining the Mostow's concept \cite[\S 1]{Mostow1979}. We    denote by $C^\infty (X^{reg})$ (resp.  $C^\infty _0 (X^{reg})$)
   the  space of smooth functions on $X^{reg}$ (resp.  the space of
smooth functions with compact support  in $X^{reg}$). 
Note that any function $f\in C^\infty_0 (X^{reg})$ has a unique
extension to  a  continuous function $j_*f$ on $X$ by setting $j_*f(x): = 0$ if $x \in X\setminus X^{reg}$.
The image $j_* (C^\infty_0 (X^{reg}))$  is a sub-algebra of $C^0 (X)$.

\begin{definition}\label{smooth} {\it  A smooth structure} on  a pseudomanifold w.i.c.s. $M$  is a choice of   a subalgebra $C^\infty (M)$ of the algebra $C^0(M)$  of all real-valued continuous functions  on $M$    satisfying the  following  three properties.

1. $C^\infty(M)$ is a   germ-defined $C^\infty$-ring, i.e. it is the $C^\infty$-ring of  all  sections  of  a sheaf  $SC^\infty (M)$  of continuous real-valued functions (for each open set $U \subset M$ there is a  collection $C^\infty (U)$ of continuous real-valued functions  on $U$  such that  the rule $U \mapsto C^\infty (U)$ defines the sheaf  $SC^\infty(M)$, moreover, for any $n$ if $f_1, \cdots ,f_n \in C^\infty (U)$  and $g \in C^\infty (\R^n)$, then $g(f_1, \cdots, f_n) \in C^\infty(U)$ \cite[\S 1]{Mostow1979}).

2. $ C^\infty (M)_{|M^{reg}} \subset C^\infty (M^{reg})$.

3. $j_*(C^\infty _0 (M^{reg})) \subset  C^\infty (M)$. 

\end{definition}

We refer the reader to  \cite{MR1991} for the theory  of  $C^\infty$-algebras.

\begin{lemma}\label{main1}  Any  smooth structure on a pseudomanifold w.i.c.s. $M$ satisfies  the following   partial invertibility.  If $f  \in  C^\infty (M)$ is nowhere vanishing, then
$ 1/f  \in C^\infty  (M)$.
\end{lemma}

\begin{proof} Assume that  $f  \in  C^\infty (M)$ is nowhere vanishing. It suffices to show that  locally $1/f$  is a smooth function.
Since $ f\not = 0$, shrinking   a neighborhood $U$ of $x$ if necessary,
 we can   assume that  there is an open interval $(-\eps, \eps)$   which has no intersection with $f(U)$.  Now there exists a smooth function $\psi:  \R \to \R$  such that  

a)  $\psi_{| (U)} = Id$,

b) $(-\eps/2, \eps/2)$   does not intersect  with  $\psi (\R)$.

Clearly  $G : \R \to \R$ defined by $G (x) = \psi (x) ^{-1}$ is a smooth function.  Note that $ 1/f(y)  =  G (f(y))$
for all $y \in U$. This completes the proof  of our claim.
\end{proof}

\begin{example}\label{res} 1. Let $\tilde M$ be an orientable  smooth manifold with  a connected orientable boundary $\p \tilde M = L$ and $M$   obtained from $\tilde M$ by collapsing $L$ to a point, see Example  \ref{exam1}.1.   Let $C^\infty(\tilde M)$  be the canonical smooth structure on $\tilde M$.  Denote by $\pi :\tilde M \to M$ the
surjective continuous  map  which  is 1-1 on $\tilde M \setminus L$  to its image $M^{reg}$.   We set 
$$C^\infty _w (M):= \{ f \in C^0 (M)|\, \pi^* (f) \in  C^\infty (\tilde M)\}. $$
It is easy to see that $C^\infty _w (M)$ 
satisfies the conditions in Definition \ref{smooth}.  We  call ${\tilde M}$
{\it the canonical resolution of  $M$}.

2. Let $L = S^n$ and $X$ be the blowup of the point of origin $0\in \R^{n+1}$ , i.e.  $X = \{ (x, l) \in \R ^{n+1} \times \R P ^n |\, x\in l\}$.  Let $\pi: X \to \R^{n +1}= cL$ be the projection  on the first factor.  We set $C^\infty_{rp} (cL): =\{ f \in C^0 (cL)| \pi ^* (f) \in C^\infty (X)\}$. It is easy to see that $C^\infty_{rp} (cL)$ is a smooth structure according to Definition  \ref{smooth}.

3. Let $L = S^{2n +1}$ and  $X$ be a blowup of the point of origin $0\in \C^{n+1}$ and $\pi : X \to cL= \C^{n+1}$ be  the  canonical projection.
Using this resolution $(X \stackrel{\pi}{\to} cL)$  we  define another smooth structure  $C^\infty _{cp}(cL)$ on $cL$,
which   also satisfies the condition in Definition \ref{smooth}.

4. Let  $M$ be a pseudomanifold w.i.c.s. and $\tilde M $ be a smooth  manifold.  We call $\tilde M$ {\it  a  resolution of
$M$} if  there exists a continuous surjective map $\pi : \tilde M \to M$  such that the restriction of $\pi$ to $\tilde M \setminus
\pi^{-1} (S_M)$ is a smooth diffeomorphism on its image. Using the same construction as in   examples above
we define a {\it  resolvable}  smooth structure $C^\infty _{\tilde M} M$ on $M$.   We observe that there are
many non-diffeomorphic   resolutions  of  a given  conical  pseudomanifold, which lead to different smooth structures
on $M$, e.g. Examples \ref{res}.2, \ref{res}.3.

5. Let $C^\infty _1(M)$ and $C^\infty _2 (M)$ be smooth structures on $M$. Then $C^\infty _1 (M) \cap C^\infty_2 (M)$ is a smooth structure on $M$.
\end{example}


\begin{definition}\label{map} 
Let $M$ and $N$ be conical pseudomanifolds  provided with smooth structures $C^\infty (M)$ and $C^\infty (N)$ respectively.  
A  continuous map  $\sigma : M \to N$ is called  {\it a smooth map}, if $\sigma ^*(f) \in C^\infty (M)$ 
  for all  $f \in C^\infty (N)$. 
\end{definition}

\begin{remark}\label{rsmooth}   Denote by $i$ the inclusion $M^{reg} \to M$. Condition (1) in Definition \ref{smooth} implies that $i$ is a smooth map.   Since the kernel of the homomorphism $i^*:C^\infty (M) \to C^0 (M^{reg})$  is zero,
we can  regard $C^\infty (M)$ as a subalgebra of $C^\infty (M^{reg})$.  In the same way we can regard
$C^\infty _0 (M^{reg})$ as a subalgebra of $C^\infty (M)$.
\end{remark}

The existence of a  smooth partition of unity on a  smooth  space is  an important condition for the validity of many
theorems in analysis and geometry \cite{Mostow1979}, for example  it is used in the proof of Lemma \ref{comp} below. 

\begin{lemma}\label{sing}
Let $s\in S$ and let $U$ be a neighborhood of $s$. Then there exists a function $f\in C^{\infty}(M)$ such that
\begin{enumerate}
\item $0\leq f\leq 1$ on $M$;
\item $f(s)=1$; 
\item $f=0$ outside $U$.
\end{enumerate}
\end{lemma}

\begin{proof}  
Obviously, there exists $\varepsilon>0$ such that $\mathcal N_s(\varepsilon)\subset U$.  We will  construct the required function $f$ in several steps using a singular chart  $\phi_s : \Nn _s (\eps) \to cL(\eps)$.

In the first step we define an auxiliary smooth function $\chi \in C^\infty _0 ((0,\varepsilon))$. It is defined in the
following way.
\begin{gather}
\chi(a)=0\text{ for }a\in(0,\frac{1}{5}\varepsilon],\quad0<\chi(a)<1\text{ for }a\in(\frac{1}{5}\varepsilon,
\frac{2}{5}\varepsilon),\notag\\
\chi(a)=1\text{ for }a\in [\frac{2}{5}\varepsilon,\frac{3}{5}\varepsilon],\quad
0<\chi(a)<1\text{ for }a\in(\frac{3}{5}\varepsilon,\frac{4}{5}\varepsilon),\notag\\
\chi(a)=0\text{ for }a\in  [\frac{4}{5}\varepsilon,\varepsilon).\notag
\end{gather}

In the second step we define  a continuous function $\chi_M \in C^0 (M)$  by setting
$$  \chi_M (x) := \chi \circ \rho_{cL}(\phi_s (x)) \text { for } x \in \Nn_s(\eps),$$
$$\chi(x):=0\text{ for }x\in (M^{reg}\setminus \phi_s ^{-1}(L\times(0,\varepsilon)).$$
Note that $\chi_M$ is a smooth function on $M^{reg}$ with compact support, and consequently an element of $C^{\infty}(M)$. 

In the third step we define a new function $\psi\in C^0 (M)$. We set
\begin{gather}
\psi(x):=1\text{ for }x\in (M\setminus \phi_s ^{-1}(cL(\varepsilon))\text{ or  }x\in \phi_s ^{-1}((L\times (\frac{2}{5}\varepsilon,\varepsilon))\notag\\
\psi(x):=\chi_M(x)\text{ for }x\in \phi_s ^{-1} (L\times(0,\frac{2}{5}\varepsilon],\text{  and } \psi(s):=0.\notag
\end{gather}
Let us show  that on a neighborhood of any point $x\in M$ the function $\psi$ coincides with a function
from $C^{\infty}(M)$. If $x\in (M\setminus \phi_s ^{-1}(cL(\varepsilon))$ or $x\in \phi _s^{-1}((L\times (\frac{2}{5}\varepsilon,\varepsilon)))$, then on
a neighborhood of $x$ the function $\psi$ coincides with the constant function $1\in C^{\infty}(M)$. If $x\in
\phi_s ^{-1} (L\times(0,\frac{2}{5}\varepsilon])$, then on a neighborhood of $x$ the function $\psi$ coincides with the
function $\chi\in C^{\infty}(M)$. Finally on a neighborhood of the point $s$ the function $\psi$ coincides with a
constant function $0\in C^{\infty}(M)$. Consequently $\psi\in C^{\infty}(M)$, and then also $f=1-\psi\in
C^{\infty}(M)$. This function has all the required properties.
\end{proof}

\begin{lemma}\label{loc}
For every compact subset $K\subset M$ and every neighborhood $U$ of $K$ there exists a function $f\in C^{\infty}(M)$
such that
\begin{enumerate}
\item $f\geq 0$ on $M$;
\item $f>0$ on $K$;
\item $f=0$ outside $U$.
\end{enumerate}
\end{lemma} 
\begin{proof}
For each point $x\in K$ we take its open neighborhood $V_x$ in such a way that $V_x\subset V$, and we take a function
$f_x\in C^{\infty}(M)$ described in Lemma \ref{sing} (note that  Lemma \ref{sing} trivially holds for  any regular
point $x\in M^{reg}$). Finally, we take an open neighborhood
$W_x\subset V_x$ of $x$ such that ${f_x}|_{W_x}>\frac{1}{2}$. Because $K$ is compact, we can find a finite number of
$x_1,\dots,x_r$ in $K$ such that
$$
W_{x_1}\cup\dots\cup W_{x_r}\supset K.
$$
Now it is sufficient to set $f:=f_{x_1}+\dots+f_{x_r}$.
\end{proof}

\begin{lemma}\label{top}
Let $\{U_i\}_{i\in I}$ be a locally finite open covering of $M$. Then there exists a locally finite open covering
$\{V_i\}_{i\in I}$ (with the same index set) such that $\bar{V}_i\subset U_i$.
\end{lemma}
\begin{proof}
The proof is standard.
\end{proof}
\begin{proposition}\label{fin}
Let $\{U_i\}_{i\in I}$ be a locally finite open covering of $M$ such that each $U_i$ has a compact closure
$\bar{U}_i$. Then there exists a partition of unity $\{f_i\}_{i\in I}$ subordinate to $\{U_i\}_{i\in I}$. 
\end{proposition}

\begin{proof}
Let $\{V_i\}_{i\in I}$ be the same covering as in Lemma \ref{top}. Let $\{W_i\}_{i\in I}$ be an open covering such
that $\bar{V}_i\subset W_i\subset\bar{W}_i\subset U_i$. According to Lemma \ref{loc} for every $i\in I$ there exists
a function $g_i\in C^{\infty}(M)$ such that
\begin{enumerate}
\item $g_i\geq0$ on $M$;
\item $g_i>0$ on $\bar{V}_i$;
\item $g_i=0$ outside $W_i$.
\end{enumerate} 
Because $V_i\subset\supp g_i\subset U_i$ for every $i\in I$, the sum $g=\sum_{i\in I}g_i$ is well defined and
everywhere positive. Since our algebra $C^\infty (M)$ is  germ-defined,  $g$ belongs to $ C^{\infty}(M)$,  and according to the partial invertibility property in Lemma \ref{main1}  $1/g\in
C^{\infty}(M)$. Consequently, defining $f_i=g_i/g$, we obtain the desired partition of unity.
\end{proof}

\begin{corollary} \label{germ} Smooth functions on $M$ separate points on $M$.
\end{corollary}

\begin{proof} 
Let $x_1,x_2\in M$, $x_1\ne x_2$. 
We take an
$\eps$-neighborhood $\Nn_{x_2}(\eps)$ of $x_2$ such that $x_1\not\in \Nn_{x_2}(\eps)$. Then it suffices to take a function $f$ from Lemma \ref{sing} and
we have $f(x_1)=0$ and $f(x_2)=1$.
\end{proof}

 Next we would like to define a
notion of a locally smoothly contractible differentiable structure on $M$.
For this purpose we shall have to take a product $U(x)\times[0,1]$, where
$U(x)$ is an open neighborhood of $x\in M$, and endow it with a
differentiable structure. Though the product $U(x)\times[0,1]$ need not be
a pseudomanifold w.i.c.s., we can use the same
concept of a smooth structure as  Mostow  used \cite[\S 3]{Mostow1979}. 
We say that  $C^\infty(M)$ is {\it locally smoothly contractible},  if for  any  $x\in M$ there exists  an open neighborhood $U(x) \ni x$ together
with a smooth homotopy $\sigma : U (x) \times  [0,1] \to  U(x)$ joining the   identity map with the  constant map $ U(x) \mapsto x$ \cite[\S 5]{Mostow1979}.
Note that  there is a natural  smooth structure $C^\infty (U(x) \times [0,1])$   generated by $  C^\infty (U(x))$ and  $C^\infty ([0,1])$ \cite[\S 3]{Mostow1979}, more precisely, the sheaf $SC^\infty (U(x) \times [0,1])$ is generated
by $\pi^*_1(SC^\infty ([0,1]))$ and $\pi_2^* (SC^\infty (U(x))))$, where $\pi_1$ and $\pi_2$ is  the projection from
$U(x) \times [0,1]$  to $[0,1]$ and  $U(x)$ respectively. In particular, $\pi_1$ and $\pi_2$ are smooth maps.

Denote by $ C^\infty _{L\times \{ 0 \}}(L \times [0, \infty))$ the  subalgebra in $C^\infty ((L \times [0, \infty))$ consisting of functions taking constant
values  along $L\times \{ 0 \}$.  Clearly  $ C^\infty _{L\times \{ 0 \}}(L \times [0, \infty))$ is isomorphic (as $\R$-algebra) to  $C^\infty _w(cL)$.

\begin{lemma}\label{wcone} 1. A function $f(x,t)\in C^\infty (L \times [0, \infty))$  belongs to $ C^\infty _{L\times \{ 0 \}}(L \times [0, \infty))$  if and only if $f$ can be written as $f(x,t) = t \cdot g(x,t)  + c$, where $g \in C^\infty (L \times [0, \infty))$  and $c \in \R$.
 
2.  Let $C^\infty_e(cL)\subset C^\infty_w(cL)$ be  the subalgebra consisting of
all functions   $f$ on  which can be written as $f([x, t] ) = g ( tf_1 (x), \cdots, tf_k (x))$ for 
some $g \in C^\infty(\R ^k)$ and $f_i \in C^\infty (L)$. Then $C^\infty_e (cL)$ is  a locally smoothly contractible smooth structure on $cL$.

\end{lemma}

\begin{proof} 1) The ``if" assertion  in the  first statement is obvious. Let us prove the ``only if" assertion.
For any  $f \in  C^\infty (L \times [0,\infty))$ we have
$$  f (x, t) = f (x, 0)  + \int _0 ^ 1 {d f (x, t r) \over d r }\, dr  = f(x, 0) + t \int _0 ^ 1 { d f ( x, t r)\over d  (tr)}\, dr .$$
Clearly $\int _0 ^ 1 { d f ( x, t r)\over d (tr)}\, dr  \in C^\infty  (L \times [0, \infty))$. 
This proves the first statement.  

2) It is easy  to see that $C^\infty_e (cL)$ satisfies the first condition in Definition \ref{smooth}.  We observe that 
$C^\infty _e (cL)$ also satisfies the second  condition of Definition  \ref{smooth}, i.e. $j_* (f) \in C^\infty _e (cL)$
for any $f \in C^\infty _0 (cL)$, since this assertion is a consequence of Remark \ref{exotic}.4 below.
To prove the second statement  of Lemma \ref{wcone} it suffices to show that    the map
$$ F :  cL(1) \times [0,1] \to  cL , \: ([x,t], \lambda) \mapsto    [ x, \lambda t] $$
is a smooth  map.  Equivalently   we have to show that any function  $F^* ( f)$, $ f\in C^\infty _e( cL(1))$,  belongs to the germ-defined $C^\infty$-ring $ C^\infty ( cL(1) \times [0,1])$  generated by $C^\infty_e (cL(1))$ and $C^\infty ([0,1])$.
Repeating the previous argument, we can write $f([x,t]) = g (t f_1(x), \cdots, tf_k (x))$, where   $ f_i\in C^\infty (L )$  and $g \in C^\infty (\R^k)$.  Clearly $(F^*( f))([x,t], \lambda) =  g ( \lambda tf_1(x), \cdots, \lambda t f_k (x))$  can be written as a function $G (\lambda,  tf_1(x), \cdots ,tf_k (x))$, hence it belongs to  $C^\infty (cL(1)  \times [0,1])$.  
\end{proof}

Now we show a  geometric way  to construct   a nice locally smoothly contractible smooth structure   on a conical pseudomanifold $M$.

\begin{definition} {\it A Euclidean  smooth structure}  on  a  pseudomanifold w.i.c.s. $M$ is  defined by a smooth embedding $I_s : L_s \to S^l(1)\subset \R^{l+1}$  and a  trivialization $\phi_s : \Nn_s \to cL_s$ for each  $s\in  S_M$ as  follows.
Let $\hat I_s$ denote the  induced embedding of $cL_s \to \R^{l+1}$.  A continuous function
on $M$ is called {\it smooth}, if   it is smooth on $M^{reg}$ and its restriction  to $\Nn_s$ is a  pull back of a smooth function on $\R^{l+1}$ via  $\hat I_s \circ \phi_s$  for all $s$. 
\end{definition}

By composing an  embedding $I_s$ with an isometric embedding $g_{l,l+k}:S^l (1) \to S^{ l +k}(1)\subset \R^{l+k+1}$  we get another embedding $ I_{s, +k}: L_{s} \to S^{l +k}(1)$.  Denote by $\hat I_{s, +k}$ the induced 
embedding $cL_s \to \R^{l+k +1}$.  It is easy to see  that  the smooth structures  defined by $I_s$  and $I_{s, +k}$ are equivalent. This motivates us to give  the following concept.

Two smooth  embeddings $I_s ^1: L_s \to S^{k_1}(1) \subset \R^{k_1+1} $ and $I_s^2 : L_s \to S^{k_2}(1)\subset \R^{k_2 +1}$ are called {\it Euclidean equivalent}, if there exists  a diffeomorphism $\Theta: \R^{k_1+ k_2 + 2}\to \R^{k_1 + k_2 +2}$   such that  $\Theta \circ\hat  I_{s, + k_2 + 1} ^ 1  =
\hat I_{s, +k_1 +1}$.  
Two  Euclidean smooth structures are  called {\it Euclidean equivalent}, if the corresponding embeddings  $I_s$ are  Euclidean equivalent.
The embedding  $\hat I_s \circ \phi_s$ is called  {\it a smooth chart} around  singular  point $s \in S_M$.

Clearly the   canonical smooth structure  on $\R^n$  is a Euclidean   smooth structure.

\begin{remark}\label{exotic} 1. A  pseudomanifold w.i.c.s. may  have more than one Euclidean smooth structure.   For example,   conical pseudomanifolds
$cL(z={1\over 2})$  and $ cL (z=0)$  are not isomorphic, where $L(z=\theta)$ is the  circle  in  $S^2 (1) \subset \R^3(x, y, z)$ defined by the equation $ z = \theta \in  (-1, 1)$.   This is proved  by observing that the function $f(x,y,z) = (x^2 + y ^ 2 + z^2) ^{1/2}$ is smooth  on  $cL(z={1\over 2})$ but it is not smooth on $cL(z=0)$. Using the diffeomorphism  $T_\alpha : \R^ 3\to  \R^3, \, z\mapsto \alpha z, \alpha \not = 0$, we conclude that all $cL(z=\alpha)$ are diffeomorphic, if  $0 < |\alpha| <1$. Note that the ``smallest" smooth structure on $cS^1 $  is the isolated smooth structure $cL(z=0)$.

2.  Clearly any Euclidean smooth structure is  locally smoothly contractible, since   the  homotopy
$cL \times [0,1] \to  cL: ([x,t], \lambda) \mapsto  [x, \lambda .t]$ is a smooth map.

3. In the next section, see Proposition \ref{unend},  we  will show that  for any fixed $L$ there are infinitely many non-equivalent  Euclidean structures  on $cL$.  

4. Suppose that $L$ is compact and $C^\infty (cL)$ is a Euclidean smooth structure on  $cL$. Let $I_s : L \to \R^k$ is defined by  $k$ smooth functions $f_i \in C^\infty (L), i = \overline{1, k}.$ Then $\tilde  f_i(t, x) : = tf_i (x)$
are generators of  the associated Euclidean smooth structure
$C^\infty (cL)$.  Thus
$C^\infty (cL)$ is a subalgebra of the algebra $C^\infty _e (cL)$. 
\end{remark}

We say that  $C^\infty(M)$ is {\it finitely generated}, if there is a finite number of
functions $f_1, \cdots, f_k\in C^\infty(M)$ such that  any $g\in C^\infty(M)$ is of  form  $g : = \hat g (f_1, \cdots, f_k)$
for some $\hat g \in C^\infty (\R^k)$. Functions $f_1, \cdots, f_k$ are called  {\it generators} of $C^\infty (M)$. Remark \ref{exotic}.4 asserts that a Euclidean smooth structure is   finitely generated.

\begin{proposition}\label{pflaum}  Suppose that $M$ and $N$ are pseudomanifolds w.i.c.s. provided  with a finitely generated smooth structure. A continuous map $\sigma: M \to N$ is smooth, if and only
if  for each $x\in M$  there exist a  smooth chart $\phi_x :U(x) \to \R^n$, a smooth chart $\phi_{\sigma (x)}: U(\sigma(x)) \to \R^m$,  and a smooth map $\tilde \sigma : \R^n \to \R^m$ such that $\phi_{\sigma (x) } \circ \sigma =
\tilde \sigma \circ \phi_x$.  Consequently,  two Euclidean smooth structures  on a pseudomanifold w.i.c.s. $M$   are Euclidean equivalent, if and only if  they are equivalent.
\end{proposition}

\begin{proof} 1) The first assertion  of Proposition \ref{pflaum}  is a special case of Proposition 1.3.8 in \cite{Pflaum2000}, see also \cite[Proposition 1.5]{MR1991} for an equivalent formulation. 
 For the convenience of the reader we  give a proof of  this assertion, which is similar to the  proof in  the case of smooth manifolds. The ``if"  part is  clear, so we  will prove  the ``only"  part.
Let $y_1, \cdots, y_m$  be coordinate functions  on $\R^m$. By our assumption, $y_k(\phi_{\sigma(x)}\circ\sigma)  $ is a smooth function
on $U(x)$, hence there  exist  smooth functions  $f_k $ on $\R^n$ such that  $f_k (\phi_x) =  y(\phi_{\sigma_x}\circ \sigma )$ for $k = 1,m$.
Now we define  a  smooth map $\tilde \sigma : \R^n  \to \R^m$ by setting
$$ \tilde \sigma  (x) = (f_1 (x), \cdots, f_k (x)).$$
Clearly $\tilde \sigma $ satisfies the condition  of our Proposition \ref{pflaum}.1.


2) Let us  prove the  ``only if" part of   second assertion  of Proposition \ref{pflaum}. Assume that  two Euclidean smooth structures $C^\infty_1 (M)$ and $C^\infty _2 (M)$ are Euclidean equivalent. Using the existence of a smooth partition of unity (Lemma \ref{fin})   and the finiteness of $S_M$, it is easy to see that   $C^\infty_1(M)$ and $C^\infty_2 (M)$ are  equivalent, i.e.   there
exists a  homeomorphism  $ \sigma : M \to M$ such that $\sigma ^*( C^\infty _1 (M)) = C^\infty _2 (M)$. 

Now we  will prove the ``if" part of the second assertion, i.e. we  assume that there
exists  a homeomorphism $\sigma : M \to M$ such that $\sigma ^* (C^\infty _1 (M)) = C^\infty _2 (M)$.
Let $\{(I_i: L_{s_i} \to S^{l_i} \subset \R ^{l_i +1}, \phi_{s_i} : \Nn_{s_i} \to cL_{s_i})|\, s_i \in S_M\}$  be embeddings   defining $C^\infty _1(M)$. Then   $\{( I_{s_i}, \sigma \circ \phi_{s_i})|\, s_i \in S_M\}$   are  embeddings defining  $C^\infty _2(M)$.  
This proves that $C^\infty_1 (M)$ and $C^\infty _2 (M)$ are Euclidean equivalent.
\end{proof}
 
Next we introduce  the notion of the cotangent bundle  of a  stratified  space $X$, which is  similar to
the notions introduced in \cite{SL1991}, \cite[B.1]{Pflaum2000}.
Note that the germs of smooth functions
$C^\infty _x (X)$ is  a local $\R$-algebra  with  the unique maximal ideal $\sm _x$ consisting of functions vanishing 
at $x$. Set $T^*_x (X) : = \sm_x / \sm_x ^ 2$.  
Since  the following exact sequence 
\begin{equation}
0 \to \sm_x \to C^\infty _x \stackrel{j}{\to} \R \to 0
\end{equation}
split, where $j$ is the evaluation map: $j(f_x) = f_x (x)$ for any $f_x \in C^\infty  _x$, the space $T_x^*X$ can be identified with the space of K\"ahler differentials of $C^\infty _x(X)$. The K\"ahler  derivation $d : C^\infty _x (X) \to T^*_x X$ is defined as follows:
\begin{equation}
d (f_x) = (f_x -  j^{-1} (f_x (x))  +\sm_x^2,
\end{equation}
where $j^{-1}: \R \to  C^\infty_x $ is the left inverse of $j$, see  e.g. \cite[Chapter 10]{Matsumura1980}, or \cite[Proposition B.1.2] {Pflaum2000}.    We  call $T^*_x X$  {\it  the cotangent space} of
$X$ at $x$. Its dual space
$ T^Z_x X: = Hom ( T^*_xX, \R)$ is called {\it the Zariski tangent  space} of $X$ at $x$.
The union $T^*X : =\cup _{x \in X} T^* _x X$ is called {\it the cotangent bundle} of $X$.  The union $\cup _{x \in X} T^Z _x X$ is called {\it the   Zariski tangent bundle} of $X$. 

Let us denote  by $\Om^1_x (X)$ the $C^\infty _x (X)$-module   $C^\infty_x (X)\otimes_\R\sm_x/\sm_x^2$. We called  $\Om^1_x (X)$
{\it the germs of 1-forms at $x$}.  
Set $\Om^k _x (X):= C^\infty _x(X) \otimes _\R \Lambda ^k(\sm_x/\sm_x^2)$.  Then $\oplus _k\Om ^k_x  (X)$ is an exterior algebra with the following
wedge  product
\begin{equation}
(f\otimes _\R dg_1 \wedge \cdots \wedge dg_k)\wedge (f'\otimes_\R dg_{k +1}\wedge \cdots \wedge dg_l) = (f\cdot f')\otimes _\R dg_1 \wedge \cdots \wedge dg_l,
\end{equation}
where $f, f ' \in C^\infty _x$ and $d g_i \in T_x ^*M$.  

Note that the K\"ahler derivation $d : C^\infty_x (X) := \Om ^ 0 _x (X) \to \Om ^1_x (X)$ extends to the unique derivation $d: \Om ^k_x (X) \to \Om ^{k+1} _x (X)$  satisfying the Leibniz property. Namely we set
\begin{eqnarray}
d(f\otimes 1) = 1 \otimes df,\nonumber\\
d((f\otimes \alpha )\wedge  (g \otimes \beta)) = d(f\otimes \alpha) \wedge g\otimes \beta + (-1) ^{deg\, \alpha} f\otimes \alpha \wedge  d(g \otimes \beta).\nonumber
\end{eqnarray}


\begin{definition} \label{smoothf} (cf. \cite[\S 2]{Mostow1979}) A section $\alpha :X \to \Lambda ^k T^*(X)$  is  called {\it a smooth differential $k$-form}, if 	for each $x \in X$ there exists
$U(x) \subset X$ such that $\alpha (x)$ can be represented as $\sum _{i_0i_1\cdots i_k} f_{i_0}df_{i_1}\wedge
\cdots \wedge df_{i_k}$ for  some $f_{i_0}, \cdots ,f_{i_k} \in  C^\infty (X)$.
\end{definition}

Denote  by $\Om(X)= \oplus _k \Om^k(X)$ the space of all smooth differential forms   on $X$. We identify  the germ at $x$
  of 
a $k$-form $\sum _{i_0i_1\cdots i_k} f_{i_0}df_{i_1}\wedge
\cdots \wedge df_{i_k}$   with  element $\sum _{i_0i_1\cdots i_k} f_{i_0}\otimes df_{i_1}\wedge
\cdots \wedge df_{i_k} \in \Om ^k_x(X)$.
Clearly the K\"ahler derivation $d$  extends to a map also denoted by $d$  mapping  $\Om(X)$ to $\Om(X)$.
 
\begin{remark}\label{injf} Let $i^*(\Om (X))$ be the restriction
 of $\Om (X)$ to $X^{reg}$.  By Remark \ref{rsmooth}  the kernel $i^* :\Om (X)\to \Om (X^{reg})$ is zero.  Roughly speaking, we can regard $\Om (X)$ as a subspace  in $\Om (X^{reg})$. 
\end{remark}

\begin{lemma}\label{nsmf}   Let $f : M \to N$ be a smooth map between   pseudomanifolds  w.i.c.s.. Then
there is a natural map $ f^* : T^* N \to T^*M$  such that $f^* (\alpha)$ is a smooth, if $\alpha $ is smooth.
\end{lemma}

\begin{proof}  Let $f^* ( C ^ \infty _{f(x)} (N))$  be the germs of smooth functions in $f ^* (C^ \infty (N))$  at $x$.
This  defines   a map : $f^* (\Om ^0 _{f(x)} (N)) \to  \Om ^0 _x M$.  Denote by $\n _{ f(x)}$ (resp. $\sm _x$)  the maximal  ideal  in  $C ^ \infty _{f(x)} (N)$ (resp. in  $C ^ \infty _x(M)$). Clearly  $f^* (\n _{ f(x)})  \subset \sm _x $.
This induces a map $ f^* : T_{f(x)}^*  N   \to  T^*_x M$. Since $f^* (C^\infty (N)) \subset  C^\infty (M)$,   the pull back $f^* (\alpha)$ is also a smooth  differential form, if $\alpha$ is smooth. This proves  Lemma  \ref{nsmf}.
\end{proof}

The following Lemma  characterizes the singularity  of a smooth structure $C^\infty (cL)$.  For any
 pseudomanifold w.i.c.s. $M$ denote by $rk (C^\infty (M))$ the  minimal number of the generators of $C^\infty (M)$.

\begin{lemma}\label{opt}   A Euclidean  smooth structure $C^\infty  (cL)$ has no singularity, if and only if   $rk (C^\infty (cL)) = rk (C^\infty (L)) = \dim (L) +1.$ 
\end{lemma}

\begin{proof}    Assume that  $C^\infty (cL)$ has no singularity,  so there is a local diffeomorphism
$f : cL(1) \to B^{l+1} \subset  \R^{l+1}$, where  $B^{l+1} $ is a  ball  in $\R^{l+1}$. Observe that $f$  sends $L$ to $\p B^{l+1}$, we get the ``only if" assertion of Lemma \ref{opt}. Now let us prove the ``if"  assertion.   The condition $rk(C^\infty (L)) = \dim (L) +1$    holds, if and only if $L$ can be embedded in $\R ^{l+1}$ as a hypersurface, where $ l = \dim L$.  Since  $(C^\infty (cL))$  is a Euclidean smooth structure,     the cone $cL $ is  a star-shaped domain in $\R ^{l+1}$. So the smooth structure on $cL$ induced by  the embedding $cL \to \R^{l+1}$ is a smooth structure without singularity. 
\end{proof}

\section{Tangent bundles and vector fields on  a    pseudomanifold w.i.c.s. with a finitely generated smooth structure}

In this section we 
 study only  finitely generated smooth structures, so we omit the  adjective ``finitely generated", if no misunderstanding can occur.
We introduce the notion of
the Nash tangent bundle   of a  smooth  pseudomanifold w.i.c.s. $M$, the notion of the Zariski tangent bundle
of $M$,  and the notion of   the tangent bundle  of $M$, as well as
the notion  of a smooth   Nash vector  field on  $M$. Their properties has been  analyzed in Lemma \ref{zaris}  and Proposition \ref{dynamic}. We introduce the notion of characteristic classes of $M$  (Remark \ref{char}). 
 Using  the invariance  of  the tangent cone and the  cotangent space  at singular points  on $M$, we  prove the existence of infinitely many   Euclidean smooth structures  on  any conical pseudomanifold  $M$  (Proposition \ref{unend}).   

Let $M^m$ be a     pseudomanifold w.i.c.s..  Since $C^\infty(M)$ is finitely generated,  there is a smooth embedding $F : M \to \R^{l+1}$ such that $ F^* ( C^\infty (\R^{l+1})) = C^\infty (M)$.
Denote by $Gr_m(\R^{l+1})$ the set of  oriented $m$-planes  in $\R^{l+1}$.   The embedding $F$ induces 
the gaussian map $\bar F: M^{reg} \to \R^{l+1} \times Gr_m (\R ^{l+1})$ sending a  point $x$
to the pair $(x, \la {T_x M^{reg}}\ra)$. Denote by $\hat M^m$ the closure of the image
$\bar F(M^{reg})$  in $\R^{l+1} \times Gr_m (\R^{l+1})$.  
We called $\hat M^m$ {\it the Nash blowup} of $M^m$. We define  the projection
$\pi: \hat M^m \to M^m$  by setting  $\pi (x, v):  = x$.

We note that the fiber $\pi^{-1} (s)$, $ s\in S_M$, is a  closed set in $Gr_m (\R^{l+1})$.  Hence $\hat M^m$ is compact,
if $M^m$ is compact. 

We define
{\it the Nash tangent cone} $\hat T_xM^m$ at a point $x\in M$  by  setting $\hat T_xM^m: = \{ v\in \R^{l+1}|\, v\in \pi^{-1} (x)\}$. If $x$ is a regular point we have $\hat T_x M^m  = T_xM^m$.
The union $\hat TM^m:=\cup _{x \in M^m} \hat T_x M^m$ is called {\it  the Nash  tangent  bundle}  of $M^m$. The  Nash tangent bundle carries a natural topology, since $\hat TM$ is a locally closed subset in $\R^{l+1} \times Gr_m (\R^{l+1})$. Clearly the inclusion $TM^{reg} \to  \hat TM^m$ is a continuous map with respect to this topology. Let $\pi: \hat TM^m \to M^m$ denote the natural projection.
Then $\pi$ is a continuous map.   Lemma \ref{Nash} below implies that   the  tangent bundle  does not depend  on the  smooth embedding $M \to \R^{l+1}$.

Now we want to  introduce the notion of a Nash smooth vector field on $M$. For this purpose  we will provide  $\hat TM$ with a  smooth structure, that is a choice $\R$-subalgebra of ``smooth functions" in $C^0 (\hat TM)$, using the induced embedding of    $\hat TM^m$ into the  product $\R^{l+1} \times \R^{l+1}$:  $(x, v) \mapsto (F (x), v)$. 
Note that if  $K$ is a   subset  of a   space $M$ with a smooth  structure $C^\infty (M)$ then we define  a  continuous function $f$ on $K$ to be smooth
($ f\in C^\infty (K)$) if  $f$ is the restriction of  $\tilde f \in C^\infty (U(K))$ to $K$, where $U(K)$  is an  open neighborhood of $K$ in $M$ \cite[p.16]{MR1991}. If $K$ is locally closed  and $M$ is finitely generated, then $C^\infty (M)$ is finitely generated. If $K$ is closed then $C^\infty (K) = C^\infty (M) _{| K}$  \cite[p. 20]{MR1991}. 

By Proportion \ref{pflaum} the projection $\pi: \hat T M^m \to M^m$ is a smooth map, since it is the restriction of
the smooth  projection $\R^{l+1} \times \R^{l+1}  \to \R^{l+1}$. 

We call a smooth section  $V:M \to \hat TM$ {\it a smooth  Nash vector field} on  $M$. By Proposition \ref{pflaum}.1 a section $V$ is a smooth  Nash vector field, if and only if
$F_*\circ V: M \to  T\R ^{l+1} = \R^{l +1} \times \R^{l+1}$ is a smooth map, where $F_* : \hat TM \to  T\R^{l+1}$ is the inclusion.

\begin{example}\label{exn} We  consider the smooth  pseudomanifold w.i.c.s. $cL(z={1\over 2})$ in Remark \ref{exotic}.1. 
It is easy to  that the  Nash blowup of $cL(z={1\over 2})$ is diffeomorphic to  the cylinder $S^1 \times \R$. The Nash tangent space $T_OcL(z={1\over 2})$ is  the cone over $\R^ 2 \setminus ( B^ 2)$, where $B^2 $ is the open  disk on $\R^2$ whose boundary $\p B^2$ is $S^1 (z={1\over 2})$.
\end{example}

\begin{lemma}\label{Nash} 1. The homeomorphism  type of the Nash blowup  $\hat M^m$ of a  smooth  pseudomanifold w.i.c.s.
$M^m$  does not depend on the choice of  a smooth embedding $F : M^m \to \R^{l+1}$. 

2. Let $f : M \to N$ be a smooth map. Then  the differential map $Df: TM^{reg}  \to TN^{reg}$ extends naturally to a smooth
map $Df: \hat TM \to \hat TN$.
\end{lemma}

\begin{proof} 1.  Let $F_1 : M^m \to \R^{l_1 +1}$ and $F_2 : M^m \to \R^{l_2+1}$ be  two smooth embeddings. Let $s\in S_M$.
The maps $F_1 \circ F_2 ^{-1}: F_2 (M) \to  F_1 (M)$  and $ F_2 \circ F_1 ^{-1} : F_1 (M) \to F_2 (M)$ are smooth maps, hence the argument in the proof of Proposition \ref{pflaum} yields that,  there are  smooth  maps  $\sigma _{12} : \R ^{l_1 +1} \to \R^{l_2 +1}$,  $\sigma _{21} : \R^{l_2+1} \to \R^{l_1 +1}$ such 
that  $(F_2 \circ F_1 ^{-1})_{| U(s)}  = (\sigma _{12}) _{| F_1 (U(s))}$  and $ (F_1\circ F_2 ^{-1})_{| U(s)} = (\sigma _{21}) _{|F_2 (U(s))}$  for  some small neighborhood $U(s)$  of $s$. The smooth maps $\sigma _{12}$
and $\sigma_{21}$ lift 
to  smooth maps  $\tilde \sigma _{12} : \R^{l_1 +1} \times Gr_m (\R^{l_1 +1}) \to \R^{l_2 +1} \times Gr_m (\R^{l_2 +1})$ and
$\tilde \sigma _{21} : \R^{l_2 +1} \times Gr_m (\R^{l_2 +1}) \to \R^{l_1 +1} \times Gr_m (\R^{l_1 +1})$.   These maps induce a  map $h_1 : \hat F_1(U(s)) \to \hat F_2 (U(s))$ and  a map $h_2 : \hat F_2 (U(s)) \to F_1 (U(s))$  such that  $h_1 \circ   h_2 = Id_{| \hat F_2(U(s))}$ and $h_2 \circ h_1 = Id _{| \hat F_1 (U(s))}$.  Hence $h_1$ and $h_2$ are homeomorphisms. This proves the first assertion of Lemma \ref{Nash}.

2. The second assertion follows directly from the construction of $\hat TM$.
\end{proof}

\begin{remark}\label{char} Let us imitate the Mather construction of characteristic classes  for singular algebraic varieties \cite[\S 2]{MacPherson1974}  to pseudomanifolds w.i.c.s. using the Nash blowup. Let $T\tilde M$ denote the  restriction of the tautological  bundle $V^m$ of the  Grassmanian $Gr_m (\R^{l+1})$ to $\hat M$ (more precisely
$T\tilde M = ( i \circ \pi) ^ *  V^m$, where $\pi : \R^{l+1} \times Gr_m (\R^{l+1}) \to Gr_m (\R^{l+1})$ is the projection, and
$ i : \hat M \to  \R^{l+1} \times Gr_m (\R^{l+1})$ is the embedding).  We set
$$  char ( M) : = \pi _* ( Dual ( char  (  T\tilde  M))),$$
where  Dual denotes the Poincare  duality  map defined  by capping   with the  fundamental  homology class.
It is easy to see that this definition  is well-defined and  it satisfies functorial properties of characteristic classes. 
\end{remark}

We  define  {\it the  tangent cone $T_x M$} as  the subset in $\hat T_x M$ consisting  of vectors  of the form  $\dot \gamma (0)$, where   $\gamma(r): [0, \infty) \to M$ is a smooth curve (ray)
 such that $\gamma (0) = x$.  Clearly the tangent cone  $T_xM$  at a regular point coincides with  the  tangent space
 $T_x M^{reg}$.  {\it The tangent  bundle $TM$} is defined as the union $\cup _{x \in M}T_xM$. It is a  closed subset of $\hat TM$, hence it has  the natural induced
 smooth structure, see the explanation before Example \ref{exn}. 
 
\begin{example}\label{etcone} Let $\gamma(r)= [\alpha (r),\beta (r)]$ be a  smooth curve (interval) on $cL$  with  $\alpha (r) \in L$ and $\beta (0) = 0$.   We provide $cL$ with  a Euclidean smooth structure,  using the  natural embedding of  $cL \to \R^{l+1}$  as a cone over smooth submanifold $L\subset  S^l(1) \subset \R^{l+1}$  that sends $[x, t]$ to   $x. t\in \R^{l+1}$, here $t\in  \R$ acts on $\R^{l+1}$ by multiplication. By Proposition \ref{pflaum}.1, $\gamma(r)$ is smooth iff $\alpha (r) \cdot \beta (r) $ is a smooth curve in $\R^{l+1}$.  Since $\beta (r) =0$, we get  $\dot \gamma (0) = \beta ' (0)\alpha (0)\in \R^{l+1}$.  Thus $T_s cL= \cup _{x \in L} \la \p_t (x) \ra _{\otimes \R}$. 
\end{example}
 
We define  {\it the degree  of flatness of  the tangent cone $T_sM$} as the number of connected components of  the subset
$\bar T_sM := \{ v \in T_s M|\:  - v \in T_s M\}\setminus \{0\}$  of  {\it flat tangent vectors $v$}.   Clearly,
the  collection of degrees of  flatness   of the  tangent cones at singular points $s \in S_M$ is  a diffeomorphism invariant of  $M$.   As an application of the degree of flatness, we    prove the following
 
\begin{proposition}\label{unend}  For any   pseudomanifold w.i.c.s. $M$ there exist  infinitely many   Euclidean smooth
structures on $M$.
\end{proposition}

\begin{proof}  It suffices to show that  there is a  smooth structure on $cL$ with any given  degree of flatness.
 First we  embed $L \to S^l(1) \cap  \{ x\in \R^{l+1}|\:  x_{l+1} = 1/2\}$, so that  the degree of flatness of $T_scL$ is zero.   Now pick  $k$ points $ x_1 , \cdots,x_k \in L\subset \R^{l+1}$.  Clearly $-x_1, \cdots,  - x_k \in S^l(1)$.
It is easy to   construct a  new embedding $L \to S^{l+1}$ such that $x_i , -x_i \in L$ for all $i =1, k$.
Moreover,  a careful construction of this  new embedding can be made  so that  $L \cap  (-L) = \{ x_1, -x_1, \cdots, x_k, -x_k\}$. 
This completes the proof of Proposition \ref{unend}.  
\end{proof}

Let $f: N \to M$ be a smooth map  between   pseudomanifolds w.i.c.s..
Denote by $f^* (\hat TM)$ the fiber product (the pullback) of $f$ and $\pi : \hat TM \to M$: $f^* (\hat TM) = \{ (x,v)\in N \times TM|\, f(x) = \pi(y)\}$. A section $s :  N \to 
f^ * (\hat TM)$ is called  smooth, if the  decomposition $i \circ s$ is a smooth map $N \to \hat T M$, where $i : f^*(\hat TM) \to TM$ is the natural
map. It is also called {\it a smooth Nash vector field
along a map $f$}. A special case of this  concept is the   notion of a  smooth   Nash vector field  (along the identity map).

Let $g_t$ be a smooth  family of diffeomorphisms of a smooth  pseudomanifold w.i.c.s. $M$   and $g_0 = Id$.
Assume that all  singular points  of $M$ are   non-trivial, i.e.  $\Nn_s$ is not diffeomorphic to  a  ball  for
all $s \in S_M$.
Then  any diffeomorphism $\psi_t$ of $M$ which is isotopic to the identity must leave $S_M$ fixed.
 By Proposition \ref{pflaum} ${d\over dt}_{| t = 0} \psi_t$ is a   smooth  Nash vector field $V$ on  $M$, which  vanishes at $S_M$. 
 
 A singular point $s$ is called {\it trivial}, if  $L_s$ is   the standard  sphere  and $cL_s$ is   diffeomorphic to $\R ^{l+1}$. Otherwise  $s$ is called {\it a nontrivial singular  point}.

\begin{proposition}\label{dynamic} Let $M$ be a compact  pseudomanifold w.i.c.s. provided with a Euclidean smooth structure, and $V$  a smooth Nash  vector field  on  $M$.

1. If  a singular point $s \in S_M$ is  nontrivial,  then  $V(s) = 0 $. 

2.  If $V(s) = 0$ for all $s \in S_M$,   there
exists a one-parameter group of   smooth diffeomorphisms $\psi_\tau$ on $M$  such that ${d\over d\tau}_{|\tau=0} \psi_\tau (x) = V(x)$ and $\psi_0 = Id$.
\end{proposition}

\begin{proof} 1) Let $V$ be a  smooth Nash vector field on  a compact  pseudomanifold w.i.c.s. $M$.  By using a smooth partition of unity it suffices to consider the case that  $sppt (V) \subset \Nn_s$ for some  $s\in S_M$. Fix an embedding $I_s : L_s \to S^{l}(1) \subset \R^{l+1}$.  By Lemma \ref{Nash}.2  we can assume that $V$ is a vector field  on $(cL_s) \subset \R^{l+1}$.
Suppose   that $V(s) \not = 0$. Using  a linear transformation of $\R ^{l+1}$ we can assume that
$V(s) = \p x_1$.  Since $T_x  cL_s  =  T_{\lambda x} cL_s$ for all $x\in cL_s ^{reg}$ and for all $\lambda > 0$, using the compactness of the Grassmanian $Gr_m (\R^{l+1})$, $ m = \dim M$, we conclude that $\p x_1$  belongs to $T_x cL_s$ for all $ x\in L_s$. 
We note that  $T _x cL = T_x L \oplus  \la \p_t (x) \ra _{\otimes \R}$ for any  $x \in L$.  Let us  denote by $\tilde V$  the projection of $\p  x_1$ to
$TL$ with respect to the above decomposition.  Then $\tilde V$ is a smooth vector field  on $L$.  We   write 
$$ \tilde V(x) = \p x_1 + \lambda (x) \p_t(x) .$$
Let  us  denote by $|.|$ the norm  defined by the Euclidean metric on $\R ^{l+1}$. Then $|\tilde V (x) | \le |\p x_1 | = 1$. Hence  $| \lambda (x )| \le 1 $.  Denote by $\tilde x_1$ the restriction of the  coordinate function $x_1$   to $L$. Since $|x_1(x)|\le 1$ we get  $ \tilde V( \tilde x_1)(x) = 1 + \tilde x_1(x) \lambda ( x) \ge 0$.
Furthermore, $\tilde V(\tilde x_1 )(x) = 0 $ only if $\tilde x_1(x) \lambda (x ) = -1$, hence $ \tilde x_1(x) = \pm 1 = -\lambda (x)$.   Hence the  differential $d\tilde x_1$ vanishes at maximal two
points  on $L$, which are the south  pole and the north pole  of $S^l(1)$.   Now we show that $L$ is  a totally geodesic sphere in $S^ l(1)$. Denote by $W_1$ the orthogonal projection of  $\p x_1$  to $S ^l(1)$.  Clearly for all $x \in L$ we have $\tilde V(x) = W_1 (x)$, where $\tilde V$ is defined above.  Hence the   integral curves of $\tilde V$  on $L$ coincide the  integral curves of $W_1$, if they have a common point.  Note the integral curve of   $W_1$  coincides with a geodesic after reparametrization. (At a point $x\in S^l$ we  intersect  $S^l$ with the  plan $\R^2$ spanned on $\p _t (x), \p x_1$.
Clearly  the integral curve of $W_1$ through $x$ lies on this  intersection.)  Hence   $L_s$   is totally geodesic.
Thus   $s$ is a removable singularity.

2) Let $F: M \to \R^{l+1}$ be a smooth  embedding, i.e.  $F ^* ( C^\infty (\R^{l+1})) = C^\infty (M)$. We will show that there exists a  smooth vector field  $\tilde V$ with compact support on  $\R^{l+1}$ such that the restriction
of $\tilde V$ to $F(M)$  coincides with the vector field $F _*(V)$.

Since $V$ is a smooth map  from $M$ to $\hat TM$, the  argument in the proof of  Lemma \ref{pflaum}  yields that there exists a   smooth map  $\sigma: \R^{l+1} \to T\R^{l +1}\supset \hat T M$ such  that $\sigma _{|F(M)} =F_*( V)$. Using a cut off function  we can  assume that $\sigma$ has a compact support in $\R^{l+1}$, since $M$ is compact.

Now we  set $\tilde V (x) :  = ( x, \pi_2 \circ \sigma)$ for  $x \in \R^{l+1}$, where  $\pi_2 : T\R^{l+1}  = \R^{l+1} \oplus \R ^{l+1}_2 \to \R^{l+1}_2$ is the projection onto the second summand. Clearly $\tilde V$ is  a  smooth vector field on $\R^{l+1}$ such that   $\tilde V_{| F(M)} =F_* (V)$.

Let $\tilde \psi _\tau$ be the smooth diffeomorphisms on $\R^{l+1}$   generated by  $\tilde V$. We will show that $\tilde \psi _\tau (x) \in M$ for all $\tau$ and for all $x\in M$.  Note that $\tilde  V(s) =  V(s) = 0$.  Hence $s$ is a fixed point of the flow $\tilde \psi _\tau (s)$  for all $\tau >0$.  %
Next we observe that,  since $M$ is  compact, there exists a  positive  number  $\eps $ such that
for all $x \in M$ we have $\tilde \psi _\tau (x) \in   M $, if $0\le \tau \le \eps$. 
 Clearly, the restriction of $\tilde \psi_\tau$  to $M$    provides us with the required  one-parameter family of diffeomorphisms. This   proves the second assertion.
\end{proof}

Let us define {\it the Zariski tangent cone $T^Z_xM$}  at a point $x$ in a smooth  pseudomanifold w.i.c.s. $M$
by setting  $T^Z_x M :=  Hom (\sm_x/\sm_x ^ 2, \R)$. The   universal property of
the K\"ahler differentials  implies 
that $T_x ^ZM$  can be identified with the space of all $\R$-valued derivations of $C^\infty _x (M)$  \cite[26.C]{Matsumura1980}, \cite[B.1.2]{Pflaum2000}.
 If $x$ is a regular point of $M$, then  $T_x^ZM = T_xM = \hat T_x M$.

 Now we compare   the Nash tangent cone and the Zariski tangent cone  at a given  singular  point $s \in  S_M$.  Without loss of generality we can assume that $M = cL \subset \R^{l+1}$.   Let $V \in \hat T_scL$. We set,  for $f \in C^\infty_s (cL)$,
 $$ V ( f) _s: =  V(\tilde f)_{s},$$
where $\tilde f  \in  C^\infty_x ( \R^n)$  such that the  restriction of $\tilde f$ to $cL$ is  $f$.  
By the definition of the Nash tangent cone there exists a sequence $x_n \in cL^{reg}$ such that $V(s) = \lim _{ x_n \to s} V(x_n)$, where $V(x_n) \in T_{x_n} cL^{reg}\subset \R^{l+1} \times \R^{l+1}$. Then  
$$ V(\tilde f) _s = \lim _{x_n \to s} V (\tilde f)_{ x_n} =\lim _{x_n \to s} V ( \tilde f_{|M^{reg}})_{ x_n}= \lim _{x_n \to s} V ( f)_{ x_n}.$$
Thus the above expression  $V(f)_s$
does not depend on the choice of  $\tilde f$. 
This defines a map $i:  \hat T_s M \to T_s ^Z M$. 
 
\begin{lemma}\label{zaris} 1. Let $\dim \sm_s / \sm_s ^ 2 = k$. Then there exist a neighborhood $\Nn_s (\eps)$  and
a  smooth embedding $\psi _s :\Nn_s (\eps) \to  \R^k$, i.e.  $\psi ^* (C^\infty ( \R^k)) = C^\infty (\Nn_s)$.

2. Assume that  the smooth structure on $M$ is Euclidean.  Then the  Zariski tangent cone is generated by  the Nash tangent cone, i.e. any element in $T^Z_s M$ is a linear combination of elements in $i(\hat T_s M)$.
\end{lemma}

\begin{proof}   1) The first assertion is a special case of \cite[Proposition 1.3.10]{Pflaum2000}. For the convenience of the reader we  sketch here  the proof of this assertion. Assume the opposite, i.e. there is a 
smooth embedding $\Nn_s \to \R ^l$, where 
$$k +1 \le  l : = \min \{\dim \R ^d |\, \text { there is a smooth embedding from } \Nn_s \to \R ^d\}.$$
Choose  a neighborhood $\Nn_s (\eps)$ and $k$  functions $f_1, \cdots, f_k \in C^\infty (\Nn_s(\eps))$ such that $df_i (s)$ form a  basis  in $\sm _s /\sm_s ^2$. Let $\tilde f_i$ be
the extension  of $f_i$ to  a smooth functions on $\R^l$, whose existence follows from Proposition \ref{pflaum}.
Denote by $I$ the ideal  of smooth functions on $\R^l$ vanishing on $\Nn_s(\eps)$.   We choose  $f_{k+1}, \cdots, f_l
\in I$ such that $ d\tilde f_1, \cdots,  d\tilde f_l$  form a basis in $T^*_s(\R^l) =\tilde \sm_s / \tilde \sm_s ^ 2$.
Its follows that $f_1, \cdots, f_k :\Nn_s (\eps)\to \R^k$ is a smooth embedding. This proves the first  assertion.

2) It  suffices to prove this assertion for $M = cL \subset \R^{l+1}$. We first  show  that there exists a smooth embedding  $L \to S^{k-1}  = S^{l} \cap \R^k$ such that $\la \hat  T_s cL\ra_\R =  \R^{k} \subset \R^{l+1}= T_s \R^{l+1}$, where  $\la \hat  T_s cL\ra_\R$ is the linear span  of $\hat  T_s cL$ in $T_s \R^{l +1} = \R^{l+1}$.  
Let us denote   the linear span $\la\hat T_s cL\ra_\R$ in $ \R^{l+1}$  by $\R^ k$. Let $\{\alpha_i |\, i \in 1, n-k\}$ be linearly independent 1-forms on $\R^{l+1}$  annihilating
the subspace $\R^k$. Since  $\alpha_i$ annihilates any tangent vector   in $TL_s $  for $i\in [1, n-k]$, it follows that $L_s \subset \R^k$. Hence $L \subset S^{k-1} = S^l \cap \R^k$.
Since
$\alpha_i$ also annihilates the  radial vector field  $\p _t(x)$, we conclude that $cL_s \subset \R^k$. 

It follows  that the map $i : \la \hat T_s cL\ra _\R = \R^k \to T^Z _scL$ is surjective,  since 
$T^*_s \R ^k \to T^* _s cL$ is a surjective map.  Hence   Lemma \ref{zaris}  follows.  
\end{proof}

  

\section{Symplectic pseudomanifolds w.i.c.s.}

 In this section we introduce the notion of a conical symplectic form on a pseudomanifold $M$ w.i.c.s. We provide many known examples  of  symplectic  pseudomanifolds  w.i.c.s.  (Example \ref{exams1}). We prove
 that  any   conical symplectic form is smooth   with respect to some  Euclidean smooth structure $C^\infty(M)$  (Corollary   \ref{smoothom}).  In particular it  is smooth with respect to the smooth structures $C^\infty _e (M)  \subset C^\infty _w (M)$.
 We also show the existence and uniqueness up to homotopy of  a $C^1$-conical  Riemannian metric compatible  with
 given conical smooth symplectic structure  (Lemma \ref{comp}). We compare our concept with   some existing concepts  (Remark \ref{SLP}).  Finally we  show that the  Brylinski-Poisson homology  can be defined on a  symplectic pseudomanifold  w.i.c.s. $M$, if the conical symplectic form is compatible with a Poisson smooth  structure. Moreover,  its Brylinski-Poisson homology  groups are isomorphic to the deRham cohomology groups of $M$
 with  the reverse  grading  if the conical symplectic form is  also smooth with respect to the compatible  Poisson smooth structure  (Lemma \ref{bryl} and  Corollary \ref{hodge}).  We show non-trivial examples of  these symplectic-Poisson smooth  pseudomanifolds w.i.c.s.  (Remark \ref{SLP}).

 \begin{definition} A   pseudomanifold w.i.c.s. $M^{2n}$  is called
{\it conical symplectic}, if    $M^{reg}$ is provided with a  symplectic form  $\om$   and for each
singular point $s\in S_M$ there exists a singular chart $(\Nn_s, \phi_s, cL_s(\eps_s))$ such that the restriction of
$\om$ to $\Nn_s$ has the form  $\phi _s ^* ( \bar \om)$ with
\begin{equation}
\bar \om([z,t]) = t^2 \hat\om(z)  +  tdt \wedge \alpha (z),
\label{scone}
\end{equation}
where $\hat\om \in \Om ^2 (L_s)$ and  $\alpha  \in \Om ^1 (L_s)$.
\end{definition}

Sometime we also denote by $\om(\alpha)$ the  symplectic form defined by (\ref{scone}).
We call  $\om(\alpha)$  {\it a conical symplectic form}.

\begin{remark}\label{concave} 
1. Taking into  account $d\om  = 0$, formula (\ref{scone}) implies that $d\alpha = 2\hat \om$.
Hence $\alpha$ defines a contact structure on $L_s$, since $\bar \om ^ n =  t^{2n-1}\hat \om ^{n-1} dt \wedge \alpha\not=0$. Thus we can write $\bar\om ={1\over 2} d (t^2 \alpha)$. 

2. Let $V$ be the radial vector field on $cL_s$ such that $V(z,t) = t\p_t$. Then  we have
\begin{equation}
t^2 \cdot \alpha =  V(z,t)\rfloor  \bar \om,
\label{defc}
\end{equation}
\begin{equation}
\Ll_V ( \bar\om) =  d (t^2  \alpha) = t^2 d\alpha +  2t dt \wedge \alpha = 2\bar \om.
\label{lfield} 
\end{equation}
It follows  that  $M\setminus (\cup _{s\in S_M} \Nn_{s})$ is a  symplectic manifold with concave boundary. (Recall that a boundary $\p M$ of  a  symplectic manifold  $(M, \om)$ is called concave, if there exists a  vector
field $X$ defined near $\p M$ and pointing  inwards such that $\Ll_X \om = \om$, see e.g. \cite{McDuff1991} or \cite{EH2002}.) Such a vector field $X$ is called {\it a Liouville vector field}.
\end{remark}

\begin{example}\label{exams1} 1. Let $\alpha_0 $ be the restriction of  1-form $\sum_{i =1} ^{k+1}(x_idy_i - y_idx_i)$ on $\R ^{2k +2}$ to the sphere $S^{2k+1}(1)\subset \R^{2k+2}$.  Then the standard symplectic form $\om_0 = \sum _{i =1} ^{k+1} dx_i \wedge dy_i$ on $\R^{2k+2}$ can be  written as in formula (\ref{scone}) with $L_s =  S^{2k+1}(1)$.  Hence, a  symplectic manifold with $m$ marked points $s_i, i =\overline{ 1,m},$  is  conical symplectic.

2. Let $G$ be a finite group  of $U(n)$ acting freely on $S^{2n-1} \subset (\R^{2n}, \om, J)$. Then the  quotient  $\R^{2n}/G$
is a  conical symplectic manifold $cL$ with isolated singularity at $0$, where $L = S^{2n-1}(1)/G$. Using  (\ref{defc}) we  observe that
the contact form $\alpha $ on $S^{2n-1}$ is invariant under the action of $G$, since $G$ preserves  $\bar \om = \om _0$, the  vector field $V(z, t)$ and $\hat \om = (\om _0)_{ |  S^{2n-1}}$.

3. Let $H : =\{ z\in \C^{n+1} |\, Q(z) = 0\}$ is a hypersurface in $\C^{n +1}$, where $Q(z)$ is a homogeneous polynomial
such that  the projectivization  $P (H):= \{  z\in \C P ^n |\, Q( z) = 0\}   $ is a nonsingular   hypersurface.
Then  $(H, (\om_0) _{| H})$   is a   symplectic cone  $cL$,  whose   base $L\subset S^{2n+1}$ is a $S^1$-fibration over $P(H)$  equipped with the standard contact form $\alpha = (\alpha_0)_{|L}$.  A particular case  with $H = Q_3$ has been
considered  in \cite{Chan2006}.  

4. A  slightly different  example is the  closure $\bar \Oo_{min}$  of a  smallest non-zero nilpotent orbit
$\Oo_{min}$ of the  adjoint action on a simple complex  Lie algebra $\g$ \cite[2.6]{Beauville2000}, \cite{Panyushev1991}. The regular stratum $\bar \Oo_{min} ^ {reg}= \Oo_{min}$ is provided with the Kirillov-Kostant symplectic form, which  can be  checked easily that it is  a conical symplectic form, see also \cite[Example 3.6]{Le2010}.  Clearly $\bar \Oo_{min}$ is  a complex cone
over  the smooth variety $\Oo_{min} /\C^*\subset  P (\g)$. We note that a complex manifolds $M^{2n}_\C$ with a holomorphic symplectic form $\om^2$
carries a  real symplectic form $\tilde \om^2 =Re\, (\om ^2) + Im\, (\om^2)$.

5. Any symplectic manifold  $(M,\om)$   with concave  boundary  $\p M$ extends  to a conical symplectic manifold
with one singular point by  attaching to $M$ a  symplectic closed cone $\bar c \p M$  as follows. 
Define  a symplectic  form on $c \p  M$ by (\ref{scone}), see also Remark  \ref{concave}.2. Then we  glue  $\bar c \p M$ with  $(M, \om)$ using  Darboux's theorem, which states that  a symplectic neighborhood $(U(\p M), \om_{|U(\p M)})$ of $\p M$
 is symplectomorphic to  $(\p M \times (1-\eps,1+\eps),\om(\alpha)) $, where $\om(\alpha)$ is defined by  (\ref{scone}), see also  \cite[exercise 3.36]{McS1998}.
 
6. Any  contact 3-manifold $(M^3, \xi)$ is a  concave  boundary  of some symplectic manifold  $(M^4,\om)$ \cite[Theorem 1.3]{EH2002}.  Attaching  a symplectic cone to $(M^4, \om)$  as above we get a   conical symplectic pseudomanifold.
\end{example}

Assume that $L_s$ is a singularity link of a conical symplectic pseudomanifold w.i.c.s. and $\alpha$ is defined by (\ref{scone}).

\begin{lemma}\label{symex} There exists a smooth embedding $I_s:(L_s, \alpha) \to  (S^{2k+1}(1),  \alpha _0)$ such that
$I_s ^*(\alpha_0) = \alpha$, if $k \ge n(n-1)$, where $n-1 = \dim L_s$.
This embedding gives rise to a  symplectic embedding $\hat I_s : (cL_s^{reg} , \om (\alpha)) \to (\R^{2k+2}, \om_0)$.
\end{lemma}

\begin{proof}  
Let $(L^n, \alpha)$ be a smooth $n$-dimensional manifold equipped with a $1$-form $\alpha$.
Using the Nash trick  we can find   an open covering $A_i$ on $L^n$: 
\begin{equation}
 L^n = \cup _{i =0} ^{n}  A_i \label{A.2}
 \end{equation}
such that each $A_i$ is  the union of  disjoint open balls $D_{i,j}$, $j=1, \dots, J(i)$ 
on $M^n$.
(Pick a simplicial decomposition of $L^n$ and construct $A_i$ by the induction on $i$.  
Let $D_{0,j}$ be a small coordinate neighborhood of the $j$-th vertex.  
We may assume that they are mutually disjoint.  
Set $A_0=\cup_{j=1}^{J(0)} D_{0,j}$.  
Suppose that $A_0, \dots, A_i$ are defined.  
Let $D_{i+1,j}$ be a small coordinate neighborhood, which contains  
$S^{i+1}_j \setminus \cup_{\ell=0}^i A_{\ell}$, 
where $S^{i+1}_j$ is the $j$-th $i+1$-dimensional simplex.  
We may assume that they are mutually disjoint.  
Set $A_{i+1}=\cup_{j=1}^{J(i+1)} D_{i+1,j}$.  
Hence we obtain desired open sets $A_0, \dots, A_n$.)   

Let $\{ \rho_i \}$ be a partition of unity on $L^n$ subordinate to the
covering $\{ A_i \}$. We write $\alpha(x) =\sum_{i =0}^{n} \rho_i(x) \cdot \alpha$.  
Clearly
the form $\alpha _i = \rho_i (x) \cdot \alpha$  has support on $A_i$.

Let $ \gamma_n: = \sum_{j=1}^{n}  y^j  dz^j$ be a smooth 1-form on $\R^{2n} (y^i, z^i)$. Let us recall

\begin{proposition}\cite[Proposition A.3]{Le2007}\label{LOA3} There is an embedding $f^i : A_i \to (\R^{2n}(y^k_i, z^k_i), \gamma_n)$ such that $f_i^* (\gamma_{n}) = \alpha_i.$  Moreover  $f^i$ can be chosen such that the image $f^i (A_i)$
lies in arbitrary small neighborhood of the origin $0 \in \R^{2n}$.
\end{proposition} 

Now we construct  an embedding $\tilde f : M\to \R^{2n (n+1)} $ by setting
$\tilde f  : = (f_0, \cdots, f_{n})$.
Clearly $\tilde f ^*(\gamma _{n(n+1)} ) =  \alpha$.  

By Proposition \ref{LOA3}  for any  arbitrary small neighborhood $O_\eps(0)$ of the origin  $0$  of $\R^{2n(n+1)}$ there exists a smooth embedding $f: L^ n \to O_\eps (0) \subset \R ^{2n(n+1)}$
such that $f^*(\sum_{k = 1} ^ {n+1} \sum_{j =1} ^{n} x^j _k dy ^j _k) = \alpha$. Here $(x^j_k, y^j_k)$ are coordinates on
$\R ^{2n(n+1)}$.  
 Now let $\alpha_1 := dz + \sum_{k = 1} ^ {n+1} \sum_{j =1} ^{2n} x^j _k dy ^i _k$ be  a contact form on  $\R^{2n(n+1) +1}$. Let $\tilde O_\eps (0) \supset O_\eps (0)$ be 
a small neighborhood  of $0\in \R^{2n(n+1) +1}$ such that there exists a diffeomorphism $\psi : \tilde O_\eps (0) \to U\subset S^{2n(n+1)+1}$ satisfying $\psi ^*(\alpha_0) = \alpha_1$. The existence of $\tilde O_\eps(0)$ together with $\psi$ follows from the Darboux theorem for contact manifolds. This completes the proof of the first assertion. The second assertion follows  from the first one, using Example \ref{exams1}.1. This completes the proof of Lemma \ref{symex}.
\end{proof}

From Proposition \ref{symex} we get immediately that  any conical symplectic pseudomanifold w.i.c.s. admits a smooth structure which
is compatible with  the given conical symplectic form, i.e.  the symplectic form  on a singular chart $\Nn_s$ is induced by  the smooth embedding  $\hat I_s : cL_s \to \R^{2l}$, defined in Lemma \ref{symex}.
Let us consider one such compatible smooth structure  on a conical symplectic manifold.  Proposition \ref{pflaum} implies immediately

\begin{corollary}\label{smoothom}  Any  conical symplectic structure is smooth with respect to  some Euclidean smooth structure
$C^\infty (M)$. In particular,  any conical symplectic structure  is  smooth with respect to the  smooth structures  $C^\infty_e (M)$,  $C^\infty _w (M)$.
\end{corollary}

Recall that a  conical Riemannian metric $g$ on a pseudomanifold w.i.c.s. $M$   is a  Riemannian metric on $M^{reg}$ such that
for all $s \in S_M$
the restriction of $g$ to a conical  neighborhood $\Nn_s$ has the form  $dt^2 + t^2 g_{|L_s}$, see e.g. \cite{Lesch1996}, \cite[6.3.4]{DLN2008}.
Further,  a Riemannian metric $g$ on $M^{reg}$ is called  compatible with a  symplectic form $\om$, if there exists an almost complex structure 
$J$ on $M^{reg}$ such that $g(X, Y) = \om (X, JY)$ is  a Riemannian metric on $M^{reg}$.  If the  resulting  metric $g$ is conical, we call
 $J$ a {\it  conical compatible almost  complex structure}. Now let $\om$ be a conical symplectic form defined by (\ref{scone}). Denote by $R$ the Reeb field on the contact manifold $(L_s, \alpha)$.
Let $J$ be   a conical almost complex structure   on $M$  compatible with   $\om$.  Since $g( \p_t , TL_s) = 0$, we get $J\p_t \in  TL_s$. Furthermore, using 
$ \om (J\p_t , \ker \om_{| L_s}) = 0$ and $\om (\p _t, R/t) = 1$, we obtain $ J(t\p_t) =  R$.  Thus any conical Riemannian metric on $(M,\om)$  compatible with  $\om$ has the form
$g  = dt ^2 +  t^2 (d\alpha ^2 + g_{|\ker \alpha})$. 

\begin{lemma}\label{comp} 1. Any    conical symplectic  pseudomanifold  w.i.c.s. $(M^{2m},  \om)$ admits a   compatible    conical Riemannian
metric  $g$, which is unique up to homotopy.

2. Let $C^\infty (M)$ be  a Euclidean smooth structure  described in Corollary \ref{smoothom}.  Any  compatible conical Riemannian metric is  also  smooth  with respect to $C^\infty (M)$ except at the singular points, where it is $C^1$-smooth. 

3. Any compatible conical  Riemannian metric  is smooth w.r.t. $C^\infty _e (M)$, $ C^\infty _w (M)$.
\end{lemma}

\begin{proof}  Let us consider the fiber bundle  $\Mm(M^{reg}, \om)\to M^{reg}$ whose fiber  $\Mm (x)$ consists of  all  Riemannian metrics  compatible with  symplectic form $\om(x)$. It is well-known that $\Mm(x) = Sp (2m)/U(m)$ is contractible. Now let us consider the subspace   $\Mm_{cone} (M^{reg}, \om)\subset \Mm(M^{reg}, \om)$ consisting of conical  Riemannian metrics.
The fiber  $\Mm_{cone} (y)$ for $y = [x,t] \in \Nn_s$ consists of Riemmanian metrics  of the form $dt ^2 + t ^2 (d\alpha ^2
+  g' _{|\ker \alpha}$), see above. This fiber is isomorphic to the space $Sp(2m-2)/U(m-1)$, so it is  contractible.
Let us take a section $s: \cup_{s\in S_M}L_s \to\cup_{s\in S_M}\Mm_{cone}(M^{reg},\om)_{|L_s}$. This section extends to  a smooth section of $\Mm(M^{reg}\setminus \cup_{s\in S_M}\Nn_s, \om)$.  It also  extends  smoothly on  $\cup_{s\in S_M}(\Nn_s, s)$ by setting
$g(y=[x,t]): = dt ^2 + t ^2 (d\alpha ^2(x)
+  g' _{|\ker \alpha(x)}$)  for $ y \in \Nn_s$. Using  a smooth partition of unity we get the existence of a  compatible conical
Riemannian metric on $M^{reg}$  by gluing these local sections.  The uniqueness up to homotopy  follows from the fact that the restriction of  two sections $g_1$ and $g_2$ of   $\Mm_{cone} (M^{reg}, \om)$  to $\cup _{s\in S_M} L_s$ are homotopic
over $\cup_{s \in S_M}L_s$, furthermore  this homotopy can be extended to  a homotopy by sections of   $\Mm_{cone} (M^{reg}, \om)$ joining $g_1$ and $g_2$ using smooth partitions of unity.   This proves the first assertion of Lemma
\ref{comp}.

Let us prove the second assertion of Lemma \ref{comp}. Choose an   embedding $I_s : L_s \to S^{2l+1}$ satisfying the condition of Lemma \ref{symex}.
Let $g$  denote the restriction of a compatible conical metric $\tilde g$ on $M^{reg}$ to $L_s$.
We note that  there exists a metric  $\bar g$  on  $S^{2l+1}$, which is compatible with $\alpha_0$,  i.e.  $\bar g (R, R) = 1$,
$\bar g (R, \ker \alpha_0) = 0$ and the restriction of $\bar \g$ to $\ker \alpha_0$ is compatible with $d\alpha_0$,
and moreover,
the restriction of $\bar g$ to $I_s ( L_s)$ coincides with  the induced metric $(I_s ^{-1})^* g$,  (note that $I_s ^{-1}$ is defined only on the image of $I_s$).
Denote by $g_0$ the  Euclidean metric  on $\R^{2l+2}$. Note that $g_0$ can be written as
$dt ^2 + t^2 (d\alpha_0 ^2 + (g_0)_ {| \ker \alpha _0})$.
Set $\hat g:= dt ^2 + t^2 (d\alpha_0 ^2 + \bar g_ {| \ker \alpha _0})$. 
We claim that $\hat g$ is a $C^\infty$-metric on $\R ^{2l+2}\setminus \{ 0\}$  and  it is $C^1$-smooth at $0\in \R^{2l+2}$. Substituting $t^2 = x_1^2 + \cdots + x_{2l+2}^2$  we reduce the proof of this assertion  to verifying that the function $ f(x):= \p _{x_1}  [ (x_1 ^2 +\cdots + x_{2l+2} ^ 2) (\bar g -g _0)(x/|x|)]$
is a continuous function on $\R^{2l+2}$.  Note that $(\bar  g - g_0)$ is a  smooth  quadratic form
on $S^{2l+1}$, so its restriction to any great circle $S^1 \subset S^{2l+1}$ is smooth.  Thus we can  reduce this  smoothness problem to the case  when $l =0$, where   the validity of our claim follows  by using
the  identity $(\arctan x)' = {1\over 1+x ^ 2}$  and expressing the   coordinates $(x_1, y_1)$ in terms of  polar coordinates $(r, \theta)$.  
This proves the second assertion of  Lemma \ref{comp}.

The last  assertion  of  Lemma \ref{comp} follows from the Nash embedding theorem  which  asserts that any Riemannian manifold admits
an isometric embedding into  sphere $S^N (1) \subset \R^{N +1}$, if $N$ is  large enough, thus the conical compatible Riemannian metric is smooth with respect to this ``new" embedding,  and taking into account the fact that
$C^\infty _e (M)$ contains any subalgebra $C^\infty (M)$  associated  with some Euclidean smooth structure on $M$.
\end{proof}

\begin{remark}\label{SLP}  Let us compare our definition of a   conical symplectic structure with the definition  of a symplectic 
structure on stratified  symplectic manifolds given by  Sjamaar and Lerman in \cite[Definition 1.12]{SL1991}.  In that paper they   define  a symplectic structure
on a stratified  symplectic manifold  $M$ to be a  subalgebra  $C^\infty (M)$ of  the algebra  $C^0(M)$ of  continuous functions on $M$  such that   $C^\infty (M)$ is  equipped with a Poisson bracket with the following property.  
The restriction of $C^\infty(M)$ to each  smooth   symplectic stratum  $S$ of $M$ is a Poisson subalgebra of the   Poisson algebra  of smooth functions on $S$. 

Let us denote
 by $G_{\om_0}$ the following linear bivector 
$$ G_{\om_0} = \p y_1 \wedge \p x_1 + \cdots  + \p y_n \wedge \p x_n.$$
It is known that $G_{\om_0}$ does not depend on the choice of  symplectic coordinates $(x_i, y _i)$ on $\R ^{2n}$  \cite[ \S 1.1]{Brylinski1988}. Given a symplectic  form $\om$  on  a pseudomanifold  w.i.c.s. $M$,  $\om$ defines a Poisson structure on $C^\infty(M)$ by the formula $\{ f, g\}_\om := G_{\om} ( df \wedge dg)$, if and only if
$G(\om)$ extends to a smooth section of $\Lambda ^2 T^Z (M)$ (i.e. $i(G_\om)$ sends  smooth  differential forms  to
smooth differential forms).   A smooth structure  $C^\infty (M)$ equipped with  such  a smooth section $G(\om)$
is called {\it a compatible Poisson smooth structure}.
Examples of   conical symplectic manifolds with  a compatible Poisson smooth structure are the  quotient $(M,\om)/G$ with isolated   singularity where $G$ is a compact  subgroup  of $Sym (M, \om)$, and   certain singular symplectic  reductions \cite{SL1991}, \cite{LMS1993}, see also a detailed explanation in \cite[Example 3.4]{Le2010}. 
 Another example  of a compatible smooth Poisson structure on a conical symplectic manifold is  a  resolvable smooth structure on the closure $\bar \Oo_{min}$ of an even minimal nilpotent orbit $\Oo_{min}$ in  complex semisimple Lie  
algebras, see Remarks \ref{exams1}.3  above. A detailed explanation is given in   \cite[Example 3.6]{Le2010}. 
\end{remark}

We end this section by introducing  the notion of  the  symplectic homology  (also called Brylinski-Poisson homology) on a conical  pseudomanifold with a compatible smooth Poisson  structure $C^\infty (M)$.   
Let $\Om^p (M^m)$ be the space of all smooth differential p-forms on $M$. Then  $\Om (M)=   \oplus_{p = 0} ^m\Om ^p (M^m)$. By Remark \ref{injf} $i^* (\Om (M)) \cong \Om (M)$ is a subalgebra in $\Om (M^{reg})$.  

We consider the {\it canonical complex}
$$\to \Om ^{n+1}(M) \stackrel{\delta}{\to}  \Om^n (M) \to ...  $$
where $\delta$ is  a  linear operator  defined as follows. Let $\alpha \in \Om (M)$  and $\alpha = \sum_j f_0  ^j df_1 ^j \wedge  df_p ^j$ be a local representation of $\alpha$ as in  Definition \ref{smoothf}.  Then we set  (cf. \cite{Kozsul1985},  \cite[Lemma 1.2.1]{Brylinski1988})
 $$ \delta (f_0 df_1 \wedge \cdots df_n): = \sum_{i =1}^n (-1) ^{i+1} \{ f_0 , f_i\}_\om df_1 \wedge \cdots \wedge  \widehat{df_i} \wedge  \cdots \wedge df_n $$
  $$ + \sum_{1\le i < j \le n} f_0 d \{ f_i, f_j\}_\om \wedge df_1 \wedge \cdots \wedge \widehat {df_i} \wedge  \cdots \wedge \widehat {df_j} \wedge \cdots \wedge df_n.$$

\begin{lemma}\label{bryl}  1) We have  $\delta  =  i(G_\om)\circ d - d \circ  i(G_\om)$. In particular  $\delta$ is well-defined.

2) $\delta^2 = 0$. 
\end{lemma}

\begin{proof} 1) The first assertion of Lemma \ref{bryl} has been proved for case of a smooth Poisson manifold $M^{reg}$ by Brylinski
in \cite[Lemma 1.2.1]{Brylinski1988}.   Since both $\delta$ and  $i(G_\om)\circ d - d \circ  i(G_\om)$
are local operators  and preserve  the  subspace $i^*(\Om (M) )\subset \Om (M^{reg})$,  Lemma \ref{bryl}.1
follows  from \cite[Lemma 1.2.1]{Brylinski1988}.

2) To prove the second  assertion of Lemma \ref{bryl} we note that $\delta ^ 2 (\alpha) ( x) = 0 $ for all $x \in M^{reg}$, since
$\delta$ is  local operator by   the first assertion.  Hence $\delta ^2  (\alpha) (x) = 0$ for all $x \in M$. 
\end{proof}

 We denote by $*_\om$ the symplectic star operator
$$*_\om: \Lambda  ^p  (\R ^ {2n}) \to \Lambda ^{2n-p} (\R^{2n}) $$
satisfying 
$$\beta \wedge *_\om \alpha = G^k (\beta, \alpha)vol,  \text{ where  } vol = \om ^n/n!.$$

Now let us consider a  conical symplectic neighborhood $(M^{2n}, \om )$  with  a compatible Poisson smooth structure.
Operator $*_\om: \Lambda ^p T_x ^* M^{reg}\to \Lambda ^{2n-p}T_x^* M^{reg}$ extends to    a linear operator $*_\om: \Om ^ p (M^ {reg}) \to  \Om ^{2n-p}(M ^{reg})$.  In particular, we have
$*_\om ( i^*(\Om ^ p ( M))) \subset  \Om ^{2n-p} (M ^{reg})$.
 
\begin{proposition}\label{sstar}  Suppose that  a conical symplectic form $\om$  on $M^{2n}$ is compatible with
a smooth Poisson structure $C^\infty (M^{2n})$.  If $\om$ is also smooth w.r.t. $C^\infty (M^{2n})$  then $*_\om (i^*(\Om ^k (M^{2n}))) = i^*(\Om ^{2n -k} (M^{2n}))$. 
\end{proposition}

\begin{proof}  We set 
$$\Om _A (M^{2n}) : = \{ \gamma \in  \Om ( M^{2n}) |\:  *_\om i^*(\gamma) \in i^*(\Om (M^{2n}))\}.$$

To prove   Proposition \ref{sstar}, it suffices to show that $\Om _A (M^{2n}) = \Om (M^{2n})$.
Note that the $C^\infty(M^{2n})$-module $\Om ^{2n}(M^{2n})$ is generated by $\om ^n$ since $\om^n$ is  smooth with respect to $C^\infty (M^{2n})$ and  $C^\infty (M^{reg})$-module $\Om^{2n} (M^{reg})$  is generated by $\om^n$.  Furthermore, $*_\om(i^* f) = i^*(f) i^*( \om^n)$  for any $f \in C^\infty (M^{2n})$. 
This  proves $*_\om (i^*(C^\infty  (M^{2n}))) = i^*(\Om ^{2n}(M^{2n}))$.  In particular $\Om ^0 ( M^{2n}) \subset  \Om _A ( X^{2n})$,
and $\Om ^{2n} (X^{2n}) \subset \Om _A (X^{2n})$.
 
\begin{lemma}\label{symm} We have
$$*_\om (i^*( \Om _A (M^{2n}))) = i^*( \Om _A (M^{2n})).$$
\end{lemma}

\begin{proof} [Proof of Lemma \ref{symm}] Let $\gamma \in \Om _A (M^{2n})$.  By definition $*_\om (i^*\gamma)  = \beta \in i^*(\Om   (M^{2n}))$.   Using  the identity $*_\om  ^ 2 = Id $, see e.g. \cite[Lemma 2.1.2]{Brylinski1988}, we get $*_\om \beta  = i^*\gamma$.
It follows $\beta \in i^*(\Om_A (M^{2n}))$. This proves that $*_\om (i^*( \Om _A (M^{2n}))) \subset i^*(\Om _A (M^{2n}))$. Taking into account $*_\om ^ 2 = Id$, this proves Lemma \ref{symm}.
\end{proof}

\begin{lemma}\label{A1}  1.  $\Om_A (M^{2n})$ is a $C^\infty (M)$-module.\\
2.  $d (\Om_A (M^{2n})) \subset \Om _A (M^{2n})$.
\end{lemma}

\begin{proof}[Proof of Lemma \ref{A1}]  1. The first  assertion follows from the  identity $*_\om ( i^*f(x) \phi (x)) =
i^*(f(x))\cdot *_\om i^*(\phi(x))$ for $x\in M^{reg}$, $f \in C^\infty(M)$, $\phi \in \Om^\infty (M^{2n})$ , and using the fact that $\Om (M^{2n})$ is a $C^\infty (M^{2n})$-module.

2. To prove the second assertion it suffices to  show that  for any $\gamma \in \Om _A (M^{2n})$ we have $ *_\om (i^*(d \gamma)) \in \Om (M^{2n})$.  Using
Lemma \ref{symm} we can write $i^*(\gamma) = *_\om \beta $ for some $\beta \in i^*(\Om _A (M^{2n}))$. Since $\beta \in \Om (M^{reg})$,  we can apply the identity $\delta \beta = (-1) ^{deg\, \beta  +1} *_\om d *_\om$  \cite[Theorem 2.2.1]{Brylinski1988},  which implies 
$$ *_\om i^*( d\gamma)) = * _\om d *_\om \beta = (-1) ^{deg\, \beta  +1}\delta (\beta) \in i^*( \Om(M^{2n})),$$
since $i^* \circ \delta = \delta \circ i^*$.
Hence $ d\gamma \in \Om _A (M^{2n})$.  This proves the second assertion of Lemma \ref{A1}.
\end{proof}

Let us complete  the proof of Proposition \ref{sstar}.  
Since $\Om ^1 (M^{2n})$ is a $C^\infty (M^{2n})$-module  whose generators
are    differentials $df$,  $f\in C^\infty (M^{2n})$, using Lemma \ref{A1} we obtain that $\Om ^1 (M^{2n}) \subset \Om _A (M^{2n})$.
Inductively,  we observe that $\Om ^k (M^{2n})$ is a $C^\infty (M^{2n})$-module whose generators are   the k-forms $d(\phi (x))$, where
$\phi (x) \in \Om ^{k-1} (M^{2n})$. By Lemma \ref{A1}, $\Om ^k  (M^{2n})\subset \Om _A (M^{2n})$,  if $\Om ^{k-1} (M^{2n}) \subset
\Om_A (M^{2n})$. 
This completes the proof of Proposition \ref{sstar}.
\end{proof}

From  Proposition \ref{sstar} we get immediately

\begin{corollary}\label{hodge} Suppose $(M, \om, C^\infty (M))$ is a smooth conical symplectic pseudomanifold
satisfying the  conditions  in  Proposition \ref{sstar}. The  Brylinski-Poisson  homology  of the complex $(\Om(M) , \delta)$  is isomorphic
to  the de Rham cohomology with   reverse grading : $H_k (\Om( M), \delta) = H ^{m -k} (\Om , d)$.  It is equal to the
 singular cohomology $H^{m-k}(M, \R)$, if the smooth structure is locally  smoothly contractible.
\end{corollary}

We like to mention that a theory of De Rham cohomology for    symplectic quotients has been considered by Sjamaar  in \cite{Sjamaar2005}.

\section{Concluding remarks}
\begin{enumerate}
\item   We have introduced  the notion of smooth structures with many good properties on conical pseudomanifolds. Some  of our results  has been extended to   a larger class of singular spaces, see  \cite{Le2010}.   Our    concept of smooth structures and  smooth symplectic structures    comprises
many known examples in algebraic geometry  and in the orbifold theory.

\item It would be interesting to investigate, when a smooth structure $C^\infty _{\tilde M} M$ given by a resolution of $M$
is finitely generated.

\item It would be interesting to find a sufficient condition for  the nonvanishing of   characteristic classes
of a  smooth conical pseudomanifold  $(M, C^\infty (M))$.
 
\item It  would be interesting to find a necessary  or sufficient condition   for a conical symplectic manifold
to admit a compatible    Poisson smooth structure.

\item  It would be interesting to develop  a Hodge theory   for a compact  smooth conical Riemannian  pseudomanifolds
and compare these results with those   developed  by Cheeger in \cite{Cheeger1983}.

\item It would be  interesting to find   sufficient conditions for developing   a Gromov-Witten theory on  smooth compact conical symplectic manifolds,  which may lead to  new invariants for  symplectic manifolds with concave boundary.

\end{enumerate}


{\bf Acknowledgement}  
H.V.L. thanks Dmitri Panyushev for  explaining     her his paper \cite{Panyushev1991}. 

\bigskip
\author{H\^ong V\^an L\^e},
\address{Institute of Mathematics of ASCR,
Zitna 25, 11567 Praha 1, Czech Republic}\\
\author{Petr Somberg},
\address{ Mathematical Institute,
Charles University,
Sokolovska 83,
180 00 Praha 8,
Czech Republic.}\\
\author{Ji\v ri Van\v zura},
\address{Institute of Mathematics of  ASCR, Zizkova 22, 61662 Brno, Czech Republic}

\end{document}